\documentclass{scrartcl}

\usepackage{lmodern}
\usepackage[utf8]{inputenc}
\usepackage[T1]{fontenc}
\usepackage[english]{babel}
\usepackage{csquotes}

\usepackage[backend=biber,style=alphabetic,url=false,giveninits=true,sorting=anyt]{biblatex}
\usepackage{amsmath,amssymb,amsthm}
\usepackage{mathtools,thmtools,stmaryrd,esint}
\usepackage{mathrsfs}
\usepackage{nicefrac}

\usepackage{hyperref}
\hypersetup{
    colorlinks=true,
    linkcolor=blue!60!black,
    citecolor=blue!60!black,
    urlcolor=blue!70!black
}
\usepackage[shortlabels]{enumitem}
\usepackage[nameinlink]{cleveref}
\crefname{equation}{}{}
\crefname{enumi}{}{}

\usepackage[svgnames]{xcolor}
\usepackage{graphicx}
\usepackage{longtable,tabu}
\usepackage{bm}
\usepackage{float}

\newcommand{\modif}[1]{#1}
\newcommand{\N}{\mathbb{N}}

\newcommand{\R}{\mathbb{R}}

\newcommand{\uball}{\mathbb{B}}

\newcommand{\LL}{\mathcal{L}}
\newcommand{\C}{\mathcal{C}}

\newcommand{\eps}{\varepsilon}
\newcommand{\hdm}{{\mathscr H}}

\newcommand{\one}{\mathbf{1}}
\newcommand{\dd}{\mathop{}\mathopen{}\mathrm{d}}

\newcommand{\fundef}[5]{%
\begin{array}[t]{lrcl}
#1 : & #2 & \longrightarrow & #3 \\
    & #4 & \longmapsto & #5
    \end{array}
}
\newcommand{\genrel}{\mathrel{\mathcal R}}

\DeclarePairedDelimiter\abs{\lvert}{\rvert}
\DeclarePairedDelimiter\norm{\lVert}{\rVert}

\DeclareMathOperator{\len}{length}

\DeclareMathOperator{\spn}{span}
\DeclareMathOperator{\tanspace}{Tan}

\DeclareMathOperator{\spt}{spt}

\DeclareMathOperator{\esssup}{ess\,sup}

\DeclareMathOperator{\intr}{int}
\DeclareMathOperator{\dom}{dom}
\DeclareMathOperator{\cc}{cone}
\DeclareMathOperator{\cl}{cl}
\DeclareMathOperator{\clcc}{\overline{cone}}
\DeclareMathOperator{\dist}{dist}
\DeclareMathOperator{\relint}{ri}
\DeclareMathOperator{\nop}{N\!}
\DeclareMathOperator{\bd}{bd}

\numberwithin{equation}{section}

\declaretheorem[name=Theorem,within=section]{thm}
\declaretheorem[name=Lemma,numberlike=thm]{lem}
\declaretheorem[name=Proposition,numberlike=thm]{prp}
\declaretheorem[name=Corollary,numberlike=thm]{cor}

\declaretheorem[name=Definition,numberlike=thm,style=definition]{dfn}

\declaretheorem[name=Remark,numberlike=thm,style=remark]{rmk}

\declaretheorem[name=Example,numberlike=thm,style=remark]{xmp}
\declaretheorem[name=Problem,numberlike=thm,style=remark]{pbm}

\title{A Rockafellar theorem for cyclically quasi-monotone maps: The regular non-vanishing case}
\author{Luigi De Pascale{\thanks{Dipartimento di Matematica e Informatica, Università di Firenze, Firenze, Italy. Email: \texttt{luigi.depascale@unifi.it}}}, Paul Pegon\thanks{CEREMADE, CNRS UMR 7534, Université Paris-Dauphine -- PSL, Paris, France \& MOKAPLAN, Inria Paris, Paris, France. Email: \texttt{pegon@ceremade.dauphine.fr}}}
\date{October 1, 2026}

\begin{document}
\maketitle
\begin{abstract}
In this paper, we study the connection between cyclic quasi-monotonicity and quasi-convexity, focusing on whether every cyclically quasi-monotone (possibly multi-valued) map is included in the normal cone operator of a quasi-convex function, in analogy with Rockafellar's theorem for convex functions. We provide a positive answer for locally Lipschitz, non-vanishing maps in any dimension, as well as for general multi-maps in dimension $1$. We further discuss connections to revealed preference theory in economics and to $L^\infty$ optimal transport. Finally, we present explicit constructions and examples, highlighting the main challenges that arise in the general case.
\end{abstract}

\vskip\baselineskip\noindent

\tableofcontents

\section{Introduction}

Our work deals with the relation between the notions of cyclic quasi-monotonicity and quasi-convexity. We shall start by discussing the latter, then move on to cyclic quasi-monotonicity.

A \emph{quasi-convex function} $f : \R^d \to \R\cup\{+\infty\}$ is a function whose sub-level sets $\{x : f(x) \leq \ell\}$ are convex for every $\ell \in \R$. Historically, quasi-convexity (or quasi-concavity) has been introduced in the first half of the \modif{20th} century through the interplay between mathematics and economics, specifically utility and decision theory. It seems (see \cite{guerraggioOriginsQuasiconcavityDevelopment2004}) that Von Neumann was the first one to make use of their defining properties, as a technical condition (without giving a definition) for his minimax theorem, in his monograph on game theory \cite{vonneumannZurTheorieGesellschaftsspiele1928}. The notion was properly introduced by de Finetti in his paper on \enquote{convex stratifications} \cite{definettiSulleStratificazioniConvesse1949}, where he studied their properties and their relation with totally ordered families of convex sets. The modern and rigorous study of sub-level sets of convex and quasi-convex functions was initiated by Fenchel in his 1953 monograph \cite[\S7--8]{fenchel1953convex}. We refer to Guerraggio's very interesting article \cite{guerraggioOriginsQuasiconcavityDevelopment2004} to know more about the origins of this notion.

Nowadays, quasi-convexity (or quasi-concavity) is widely used in optimization and game theory (as a common assumption for objective functionals or to apply Von Neumann's or Sion's minimax theorem), in economics (as a common assumption for utility functions in consumer theory), and calculus of variations (when supremal functionals are involved). We will elaborate later on the strong connections between the problem we address and both the theory of revealed preferences in economics and optimal transportation.

A \emph{multi-map} or \emph{multi-valued operator} $F: \R^d \multimap \R^d$ (which is nothing but a map from $\R^d$ to $2^{\R^d}$) is cyclically quasi-monotone if for every sequence of points $x_0, \ldots, x_N  \in \R^d$ satisfying $x_N = x_0$, and every $p_0 \in F(x_0),\ldots, p_{N-1} \in F(x_{N-1})$ it holds:
\[\min_{0\leq i < N} (x_{i+1}-x_i)\cdot p_i \leq 0.\]
We may trace back this notion to the papers \cite{levinQuasiconvexFunctionsQuasimonotone1995,daniilidisSubdifferentialsQuasiconvexPseudoconvex1999}. The weaker notion of \emph{quasi-monotonicity} corresponds to considering sequences only made of two points, i.e. imposing $N=1$.

Cyclically quasi-monotone operators naturally arise from quasi-convex functions $f$ through their \emph{normal cone multi-map}, defined as
\begin{equation*}
    \forall x\in \R^d, \qquad \nop f(x) \coloneqq N_{\{y:f(y)\leq f(x)\}}(x),
\end{equation*}
where $N_C(x) \coloneqq \partial \chi_C(x)$ is the normal cone\footnote{Here, $\chi_C$ stands for the convex indicator of $C$ and $\partial$ to the classical Fenchel--Moreau subdifferential operator for convex functions \modif{defined by $\partial f(x) = \{p \in \R^d : \forall y, f(y)\geq f(x) + p\cdot(y-x)\}$.}} to the convex set $C$ at $x$. Indeed, if $f$ is quasi-convex then $\nop f$ \modif{is always non-reduced to zero} and is a cyclically quasi-monotone multi-map (as stated in \Cref{normal_is_qm}). We can draw a parallel with the usual notions of convexity and cyclical monotonicity, the normal cone operator replacing the subdifferential operator of convex functions: indeed, it is well-known that $\partial f$ is cyclically monotone whenever $f$ is convex. In fact converse results exist for the classical notions of convexity and (cyclical) \modif{monotonicity}:
\begin{itemize}
    \item by \cite{poliquinSubgradientMonotonicityConvex1990,correaSubDifferentialMonotonicity1994}, for every function $f$, the Clarke--Rockafellar subdifferential $\partial f$ is monotone (even cyclically) if and \modif{only} if $f$ is convex;
    \item by the celebrated theorem of Rockafellar (proved in \cite{rockafellarCharacterizationSubdifferentialsConvex1966}, see also \cite[Theorem~24.8]{rockafellarConvexAnalysis2015}), an operator $F$ is cyclically monotone if and only if $F\subseteq \partial f$ for some lower semi-continuous proper convex function $f : \R^d \to \R\cup\{+\infty\}$.
\end{itemize}

We are tempted to wonder if such converse statements exist for the corresponding \enquote{quasi} notions, and some suitable normal cone operator replacing $\partial$. Thus, recalling the questions originally studied by de Finetti \cite{definettiSulleStratificazioniConvesse1949} and Fenchel \cite[\S7--8]{fenchel1953convex}, we highlight three classes of problems related to (quasi-)convexity and (quasi-)monotonicity:
\begin{enumerate}
    \item \emph{The characterization of sub-level sets of convex functions.} Given a totally ordered family $\C$ of convex sets (indexed or not), is there a (quasi-)convex function $f$ whose level sets belong to $\C$? If not, can we characterize those $\C$'s that are the sub-level sets of convex functions? The answer is \modif{positive} when the convex sets are closed, if we only ask $f$ to be quasi-convex (an indexed version is shown in \Cref{existence_potentials}). The answer is \modif{negative} if one asks $f$ to be convex, but some sufficient conditions were provided by de Finetti \cite{definettiSulleStratificazioniConvesse1949}, Fenchel \cite[\S7--8]{fenchel1953convex} and more recently by Rapcsák \cite{rapcsakFenchelProblemLevel2005,rapcsakErratumFenchelProblem2013} in particular. We refer to \cite{rapcsakSurveyFenchelProblem2005} for a survey on this problem, sometimes \modif{referred} to as \emph{Fenchel problem of level sets}.
    \item \emph{The subdifferential characterization of quasi-convexity.} Can we characterize, in the same way as for convexity, the quasi-convexity of a function $f$ by the quasi-monotonicity of a suitable first-order operator $Tf$, like a subdifferential or normal cone operator? This has been answered positively in \cite{ausselSubdifferentialCharacterizationQuasiconvexity1994} when $T$ is the Clarke--Rockafellar subdifferential, and partial answers have been given for different notions of normal cone operators by several authors like Aussel, Crouzeix, Daniilidis and Penot to name a few \cite{ausselNormalCharacterizationMain2000,ausselNormalConesSublevel2001,bordeContinuityPropertiesNormal1990,penotAreGeneralizedDerivatives1998}.
    \item \emph{The integration of cyclically quasi-monotone operators.} Is any cyclically quasi-monotone operator included in the graph of a first-order (of subdifferential or normal cone type) operator associated with a quasi-convex function? In other words, is there \modif{a} \enquote{quasi} counterpart to Rockafellar integration theorem known for cyclically monotone operators, for some replacement of the subdifferential operator? A partial answer has been given, under extra assumptions on the operator, with the lower subdifferential operator of Plastria by Bachir, Daniilidis and Penot in \cite{bachirLowerSubdifferentiabilityIntegration2002}. Another partial answer has been given for the normal cone operator $N$, under some regularity and non-degeneracy assumptions by Crouzeix \emph{et al.} in \cite{crouzeixConstructiveProofExistence2012}, through an \enquote{indirect} statement that is not completely equivalent.
\end{enumerate}

Our work deals with the third question, which can be set precisely as follows.

\begin{pbm}\label{main_problem}
Given a cyclically quasi-monotone multi-map $F :  \R^d \multimap \R^d$, can we build a quasi-convex function $f : \R^d \to \R\cup\{+\infty\}$, which satisfies
\[F \subseteq \nop f?\]
Such a function $f$ will be called a \emph{potential} of $F$. We shall restrict ourselves to the case of finite and lower semi-continuous potentials, as justified in \Cref{finite_lsc_potentials}.
\end{pbm}

Before stating our results and describing the structure of the paper, let us stress the connection of this problem with consumer theory in economics, more precisely to revealed preference theory, and optimal transport theory.

\subsection{Relation with revealed preference theory}

Revealed preference theory in economics, introduced by Samuelson in 1938 \cite{samuelsonNotePureTheory1938}, is a theory that postulates the rational behavior of consumers. Given $N$ bundles of goods $(x_1,\ldots, x_N) \in \R^d$ and the observation of prices $(p_1, \ldots, p_N) \in \R^d$ at which the consumer bought those goods, we may infer a \enquote{preference relation} by the following: $x_i$ is preferred to $x$, if $p_i \cdot x_i \geq p_i \cdot x$, i.e. the consumer could have chosen bundle $x$ with the same (or a lower) budget $p_i \cdot x$ but decided to buy $x_i$ instead.

In revealed preference theory, if certain axioms ensuring the consistency of the consumer’s preferences are satisfied, it is possible to deduce that the observed choices result from the maximization of a utility function $u$; that is, the preference relation $\preceq$ corresponds (in our convention) to the minimization of a (quasi-)convex function $f = -u$. One says that $f$ (or $u$) rationalizes the observed data. The axioms are called the \emph{(weak/strong/generalized) axioms of revealed preference} (WARP, SARP, GARP, respectively), and the main result in the area is the so-called \emph{Afriat Theorem}: if finitely many data $(x_1,p_1),\ldots, (x_N,p_N) \in \R^d$ satisfy (GARP), then there exists a continuous convex nonsatiated\footnote{Meaning that $f$ has no local minimum.} function $f$ rationalizing the data.

\modif{It turns out that $(x_1,p_1),\ldots, (x_N,p_N)$ \modif{satisfy} GARP if and only if the operator $F$ defined by $F(x_i) = -p_i$ satisfies the following: for every cycle $y_0, \ldots, y_N  \in \R^d$ (where $y_N = y_0$), and every $p_0 \in F(y_0),\ldots, p_{N-1} \in F(y_{N-1})$ it holds:
\begin{equation}\tag{GARP}
    (\forall i<N, (y_{i+1}-y_i)\cdot p_i \geq 0) \implies \forall i<N, (y_{i+1}-y_i)\cdot p_i = 0.
\end{equation}
This condition is stronger than cyclical quasi-monotonicity.} The rationalization condition \enquote{if $x_i$ is preferred to $x$ then  $u(x_i) \geq u(x)$} rewrites as
\[\forall x \in \{f < f(x_i)\}, \quad F(x_i) \cdot(x-x_i) < 0,\]
which implies, by nonsatiation of $f$, that
\[\forall x \in \{f \leq f(x_i)\}, \quad F(x_i) \cdot(x-x_i) \leq 0,\]
which exactly means that $F \subseteq \nop f$. In particular, Afriat Theorem solves \Cref{main_problem} in the case of an operator $F$ with finite domain. \modif{Reinhard} John may have been the first to notice the relation between cyclical quasi-monotonicity and Afriat Theorem in \cite{johnNoteMintyVariational2001}. We are actually trying to prove a slightly weaker version of Afriat Theorem but for infinite bundles. Of course, asking for concavity is too much in this case, but quasi-concavity is reasonable.

\subsection{Relation with optimal transport}

Classical optimal transport consists in the following minimization problem, given two compactly \modif{supported} probability measures $\mu,\nu$ over $\R^d$,
\[\inf \quad\left\{\int_{\R^d \times \R^d} c(x,y) \dd\gamma(x,y) : \gamma \in \Pi(\mu,\nu)\right\},\]
where $\Pi(\mu,\nu)$ is the set of probability measures on $\R^d \times \R^d$ having $(\mu,\nu)$ as marginals, and $c : \R^d \times \R^d \to \R\cup\{+\infty\}$ is a given cost function, typically $c(x,y) = \abs{y-x}^2$ (see for instance \cite{santambrogioOptimalTransportApplied2015}). When $c$ is continuous, it may be shown that $\gamma$ is optimal if and only if $\spt \gamma$ is $c$-cyclically monotone, i.e. for every $(x_0,y_0),\ldots, (x_N,y_N)$ with $(x_N,y_N) = (x_0,y_0)$,
\[\sum_{i=0}^{N-1} c(x_i,y_i) \leq \sum_{i=0}^{N-1} c(x_i,y_{i+1}).\]

The first relation between optimal transport and our problem is through the revealed preference theory discussed above. Indeed, when $c(x,y) = \log(x\cdot y)$ and $\mu,\nu$ are supported on $\R_+^d$, it turns out that $c$-cyclical monotonicity corresponds to yet another axiom of \modif{revealed} preference, the \emph{homogeneous axiom of revealed preference} (HARP), in which case finite data satisfying (HARP) can be rationalized by a positive homogeneous utility function, as shown in \cite{kolesnikovRemarksAfriatsTheorem2013}.

The second link between our problem and optimal transport lies in $L^\infty$ optimal transport, consisting in replacing the integral functional by a supremal cost:
\[\inf \quad\left\{\gamma-\esssup_{(x,y)} c(x,y) : \gamma \in \Pi(\mu,\nu)\right\}.\]
The problem has been thoroughly studied in \cite{champion$infty$WassersteinDistanceLocal2008} in the case of $c(x,y) = \abs{y-x}$ (or any power of it, since the problem is unchanged). Note that in the classical optimal transport problem, taking the quadratic cost $c(x,y) = \abs{y-x}^2$ or $c(x,y) = - x\cdot y$ is equivalent, but this is no longer the case in the $L^\infty$ framework. Besides, in $L^\infty$ theory, $c$-cyclical monotonicity is replaced with $c$-$\infty$-cyclical monotonicity:
\[\max_{0\leq i < N} c(x_i,y_i) \leq \max_{0\leq i < N} c(x_i,y_{i+1}).\]
It is known that $c$-$\infty$-cyclical monotonicity of $\spt \gamma$ no longer characterizes general solutions of $L^\infty$ optimal transport, but characterizes special solutions known as \emph{restrictable solutions} (at least when $c$ is the Euclidean distance). When $c(x,y) = -x \cdot y$, $c$-$\infty$-cyclical monotonicity rewrites as
\[\min_{0\leq i < N} x_i \cdot y_i \geq \min_{0\leq i < N} x_i \cdot y_{i+1} = \min_{0\leq i < N} (x_i \cdot y_i + x_i \cdot (y_{i+1}- y_i)),\]
thus it implies
\[\min_{0\leq i < N} x_i \cdot (y_{i+1}- y_i) \leq 0,\]
which \emph{exactly matches} the notion of cyclical quasi-monotonicity. Recovering a potential for cyclically quasi-monotone (multi-)maps appears then as an important step toward recovering some kind of replacement for Kantorovich potentials, in the context of $L^\infty$ optimal transport with $c(x,y) = -x \cdot y$. This $L^\infty$ optimal transport problem may have some application, since it corresponds to a limit Value-at-Risk functional $\text{VaR}_\alpha$ when $\alpha \to 0^+$ with correlation cost, i.e. to the maximization of $\text{VaR}_0(X\cdot Y)$ among random variables $(X,Y)$ with fixed marginals. This correlation cost with fixed marginals is considered in particular in \cite{ruschendorfSolutionStatisticalOptimization1983}.

\subsection{Main results, strategy and structure of the paper}

Our main contribution in this paper is the resolution of \Cref{main_problem} when $F : \R^d \to \R^d$ \modif{is a locally Lipschitz, single valued map which} never vanishes, as stated in \Cref{main_theorem}. The proof consists in the following strategy: build from $F$ a preference relation that has convex and closed sub-level sets, then a lower semi-continuous quasi-convex \modif{function} with these exact levels. In all generality, the relation we define is a pre-order and not a preference relation, since it is not total. Nevertheless, we also provide a solution under no extra assumption in dimension $d=1$, as well as a conditional result, stating that the problem is solved when the relation is a preference relation. Finally, we try to illustrate the difficulties ahead for the general case with a series of examples.

This paper is organized as follows. In \Cref{section_maps_potentials}, we show the identification between classes of potentials, preference relations and set families, thus recovering a potential from a preference relation, and we define a candidate preference relation for a given cyclically quasi-monotone map. \Cref{sec:1d} is devoted to the one-dimensional case. \Cref{sec:nd} deals with the $d$-dimensional case with the \modif{locally Lipschitz} regularity and non-vanishing conditions. The paper ends with \Cref{sec:examples} providing many examples.

\section{From cyclically quasi-monotone maps to quasi-convex potentials}\label{section_maps_potentials}

A \emph{multi-valued map} (or \emph{multi-map}) $F :  \R^d \multimap \R^d$ is a map from $\R^d$ to $2^{\R^d}$. Its \emph{domain} consists of all points $x$ such that $F(x) \neq \emptyset$. In our work, we shall always consider maps with \emph{full domain}, i.e. whose domain is $\R^d$. Moreover, we shall, from time to time, assimilate $F$ with its graph $ \{(x,y) \in \R^d \times \R^d : y \in F(x)\}$, thus writing $(x,y)\in F$ instead of $y \in F(x)$.

\begin{dfn}[Paths, cycles, ascending paths and cycles]
A \emph{discrete path} on $\R^d$ is a finite sequence of points $x_0, \ldots, x_N \in \R^d$. Such a path is called a \emph{cycle} if $x_0 = x_N$, and it is said \emph{$F$-ascending}\footnote{Or just ascending when there is no confusion.} if for every $0 \leq i< N$ there exists $p_i \in F(x_i)$ such that $(x_{i+1}-x_i)\cdot p_i > 0$.
\end{dfn}

\subsection{Cyclically quasi-monotone maps and quasi-convex potentials}

Let us reformulate the definition of cyclic quasi-monotonicity, given in the introduction, with the notion of $F$-ascending cycle.

\begin{dfn}[Cyclically quasi-monotone maps]
A multi-valued map $F :  \R^d \multimap \R^d$ with full domain is \emph{cyclically quasi-monotone} if there is no $F$-ascending cycle. In other words, for every cycle $x_0, \ldots, x_N  \in \R^d$ (where $x_N = x_0$), and every $p_0 \in F(x_0),\ldots, p_{N-1} \in F(x_{N-1})$ it holds:
\[\min_{0\leq i < N} (x_{i+1}-x_i)\cdot p_i \leq 0.\]
\end{dfn}

\begin{rmk}
If $F :  \R^d \multimap \R^d$ is cyclically quasi-monotone, then so are:
\begin{itemize}
\item its closure, $\cl F$, defined as its minimal closed extension (as a graph);
\item its convex cone envelope, $\cc F$, defined as its minimal extension whose image at every point is a convex cone, or equivalently
\[(\cc F)(x) = \Bigl\{\sum_{0\leq i < N} \lambda_i p_i : \forall i, \lambda_i \geq 0 \text{ and } p_i \in F(x)\Bigr\}\quad (\forall x);\]
\item its closed convex cone envelope, $\clcc F$, defined as the minimal extension which is closed and whose image at every point is a (necessarily closed) convex cone;
\item $\one_M(x) F(x)$ where $\one_M$ is the characteristic function (value in $\{0,1\}$) of a subset $M\subset \R^d$.
\end{itemize}
\end{rmk}

As said in the introduction, it turns out that cyclically quasi-monotone maps are closely related \modif{to} \emph{quasi-convex} functions, whose definition we now recall.

\begin{dfn}[Quasi-convex function and normal cone multi-map]\label{def_quasi-convex_function}
A function $f : \R^d \to \R\cup\{+\infty\}$ is \emph{quasi-convex} if its sub-level sets $\{x : f(x) \leq \ell\}$ are convex for every $\ell \in \R$. The \emph{normal cone multi-map} $\nop f :  \R^d \multimap \R^d$ of such a function $f$ is defined by
\begin{equation*}
    \forall x\in \R^d, \qquad \nop f(x) = N_{C_f(x)}(x),
\end{equation*}
where $N_C(x) \coloneqq \partial \chi_C(x)$ is the normal cone to the convex set $C$ at $x$, and $C_f(x) \coloneqq \{w\in\R^d : f(w) \leq f(x)\}$. Equivalently, for every $x\in \R^d$,
\begin{equation}
    \nop f(x) = \{p\in \R^d : \forall w, f(w) \leq f(x) \implies p \cdot(w-x) \leq 0\}.\label{normal_cone_def2}
\end{equation}
\end{dfn}

In the following proposition, we show that quasi-convex functions induce cyclically quasi-monotone operators through their normal cone operator.

\begin{prp}\label{normal_is_qm} Let $F : \R^d \multimap \R^d$ be a multi-map included in the normal cone operator of \modif{some function} $f : \R^d \to \R\cup\{+\infty\}$, namely
\[\forall x\in \R^d, \quad F(x) \subseteq \nop f (x).\]
Then $F$ is cyclically quasi-monotone.
\end{prp}

\begin{proof}
Let us consider $(x_0,p_0),\ldots (x_N,p_N)$ an $\nop f$-ascending path. Let $0 \leq i < N$. Since $p_i \cdot (x_{i+1}-x_i) > 0$ and $p_i \in \nop f(x_i)$, by \labelcref{normal_cone_def2}, we have
\[f(x_{i+1}) > f(x_i).\]
As a consequence $f(x_N) > f(x_0)$ thus $x_0 \neq x_N$ and there is no $\nop f$-ascending cycle.
\end{proof}
\modif{This result holds true for any \emph{general} function $f$ but the assumption of quasi-convexity guarantees that $\nop f(x)$ is never reduced to zero.}

Our goal is to tackle \Cref{main_problem}, i.e. to show that the normal cone multi-maps associated with quasi-convex functions are, in some sense, the \emph{only} cyclically quasi-monotone multi-maps, at least under some assumptions. As said in the introduction, let us explain that we can restrict ourselves to lower semi-continuous potentials that only take finite values.

\begin{rmk}[Finiteness and lower semi-continuity of potentials]\label{finite_lsc_potentials}
First of all, by composing $f$ with a strictly increasing map from $\R\cup\{+\infty\}$ to $\R$ (for example $\arctan (t)$), we can assume $f$ to be valued in $\R$.

Moreover, note that in \Cref{def_quasi-convex_function}, $C_f(x)$ may be replaced with its closure because
\[p \in \partial\chi_{C_f(x)}(x) \iff p \in \partial \chi_{\overline{C_f(x)}}(x),\]
thus one is tempted to look for potentials $f$ which are lower semi-continuous, ensuring that $C_f(x)$ is closed for every $x \in \R^d$. Indeed, this will be possible, since by \Cref{existence_potentials} for any given quasi-convex function $f$ we can find a lower semi-continuous function $\tilde f$ such that for every $x\in \R^d$, $C_{\tilde f}(x) = \overline{C_f(x)}$, which implies that $\nop f(x) = \nop{\tilde f}(x)$.
\end{rmk}

\subsection{Potentials, preference relations and set families}

Before tackling \Cref{main_problem}, i.e. the question of existence of a (finite and lower semi-continuous) quasi-convex potential $f$, what about uniqueness? Note that if there is a solution $f$, then $\phi \circ f$ is also a potential of $F$ whenever $\phi : \R \to \R$ is a strictly increasing function, thus there is no uniqueness. What matters is actually the binary relation $\preceq_f$ induced by $f$, defined by
\[x\preceq_f y \iff f(x) \leq f(y),\]
which remains unchanged after composition with $\phi$. It is a quasi-convex and lower semi-continuous \emph{preference relation}, a notion which we shall define just after recalling some useful vocabulary on binary relations.

\begin{dfn}[Vocabulary on binary relations]\label{vocabulary_binary_relations}
A binary relation $\mathcal R$ on a set $X$ is
\begin{itemize}
    \item \emph{reflexive} if $x \genrel x$ is true for every $x\in X$,
    \item \emph{irreflexive} if $x\genrel x$ is false for every $x\in X$,
    \item \emph{asymmetric} \modif{if for every $x,y\in X$, $x\genrel y$ and $y \genrel x$ are not simultaneously true,}
    \item \emph{antisymmetric} if $x\genrel y$ and $y \genrel x$ implies $x=y$, for every $x,y\in X$,
    \item \emph{transitive} if $x\genrel y$ and $y\genrel z$ implies $x\genrel z$ for every $x,y,z\in X$,
    \item \emph{total} (or \emph{connected}) if either $x\genrel y$ or $y\genrel x$ is true, for every $x,y\in X$.
\end{itemize}
\end{dfn}

\begin{dfn}[Pre-order and preference relation]
A binary relation $\preceq$ is a \emph{pre-order} if it is reflexive and transitive, and it is a \emph{preference relation} if it is a total pre-order. The pre-order $\preceq$ is said
\begin{itemize}
\item \emph{lower semi-continuous} if for every $x$, $\{w : w \preceq x\}$ is closed,
\item \emph{quasi-convex} if for every $x$, $\{w : w \preceq x\}$ is convex.
\end{itemize}
\end{dfn}

All the properties of the preference relation $\mathalpha{\preceq}$ can be equivalently expressed in terms of the family (indexed by $\R^d$) of subsets
 \[\mathcal C_{\mathalpha{\preceq}} \coloneqq (C(x))_{x\in \R^d},\quad\text{where}\quad C(x) = \{w : w \preceq x\}\]
 for every $x\in \R^d$. In order to build a quasi-convex potential for $F$ (up to strictly increasing composition), we may therefore try to build three different types of objects: (quotiented) potentials, preference relations or set families, defined as follows.

 \begin{dfn}[Spaces of quotiented potentials, preference relations and set families]
 We denote:
 \begin{itemize}
\item[$\mathsf{F}_d$] the set of quasi-convex lsc functions $f : \R^d \to \R$ quotiented by the relation $\sim$ defined by $f \sim g \iff f = \phi \circ g$ for some strictly increasing function $\phi : g(\R^d) \to \R$;
\item[$\mathsf{P}_d$] the set  of lower semi-continuous quasi-convex preference relations over $\R^d$;
\item[$\mathsf{C}_d$] the set of families $\mathcal C = (C(x))_{x\in \R^d}$ of totally ordered\footnote{For the inclusion.} closed convex subsets of $\R^d$ which satisfy $x\in C(x)$ and ($w\in C(x) \implies C(w)\subseteq C(x)$) for every $x,w\in \R^d$.
\end{itemize}
\end{dfn}

We shall prove that these three sets are in bijection, through the canonical maps:
\begin{equation}\label{def_bijections}
    \fundef{\Phi}{\mathsf{F}_d}{\mathsf{P}_d}{{f}/{\mathalpha{\sim}}}{\mathalpha{\preceq_f}} \quad\text{and}\quad \fundef{\Psi}{\mathsf{P}_d}{\mathsf{C}_d}{\preceq}{\mathcal C_{\preceq} = (\{w : w \preceq x\})_{x\in \R^d}},
\end{equation}
the map $\Phi$ being well-defined because two equivalent functions produce the same preference relation, as already noticed.

The only difficulty in showing that they are bijective is to invert $\Phi$, more precisely to build, starting from a quasi-convex lower semi-continuous relation $\mathalpha{\preceq}$, a lower semi-continuous quasi-convex function $f$ that satisfies $\mathalpha{\preceq} = \mathalpha{\preceq_f}$. To obtain a lower semi-continuous function $f$, we shall use the following lemma, providing a \enquote{sub-level-stable} lower semi-continuous envelope (sublsc envelope for short).

\begin{lem}[Sublevel-stable lower semi-continuous envelope]\label{sublsc_envelope}
Let $g : \R^d \to \R$. If $\{g\leq g(x)\}$ is closed for every
$x\in \R^d$, then the function $\hat g : \R^d \to \R$ defined by
\begin{equation}\label{def_function_from_set_family}
    \hat g(x) = \inf \left\{ \liminf_{n\to +\infty} g(x_n') : x_n' \to x'\text{ and } g(x') = g(x)\right\} \quad (\forall x\in \R^d)
\end{equation}
is lower semi-continuous and satisfies
\begin{equation}\label{prop_function_from_set_family}
\{g\leq g(x)\} = \{\hat g\leq \hat g(x)\}\quad(\forall x\in \R^d).
\end{equation}
\end{lem}
\begin{proof}
It is immediate to see that
\begin{equation}\label{hat-is-smaller}
    \hat g \leq g
\end{equation}
by taking $x' = x$ and $x'_n = x$ for every $n$ in the definition of $\hat g(x)$. Denoting by $\bar g$ the lower semi-continuous envelope of $g$, \labelcref{def_function_from_set_family} may be rewritten as
\begin{equation}\label{f_characterization_lsc_envelope}
    \hat g(x)
    = \inf \{ \bar g(x') : g(x) = g(x')\}
\end{equation}
and thus
\begin{equation}\label{f_constant_on_equivalent_points}
    \hat g(x) = \hat g(x') \quad\text{whenever}\quad g(x) = g(x').
\end{equation}

We first prove \labelcref{prop_function_from_set_family}. We shall successively prove the following, for every $x,y\in\R^d$:
\begin{gather}
    \label{f_increasing_step1}
    g(x) < g(y) \implies g(x) \leq \hat g(y),\\
    \label{f_nondecreasing}
    g(x) \leq g(y) \implies \hat g(x) \leq \hat g(y),\\
    \label{f_increasing}
    \hat g(x) \leq \hat g(y) \implies g(x) \leq g(y).
\end{gather}
Obviously \labelcref{f_nondecreasing} and \labelcref{f_increasing} yield \labelcref{prop_function_from_set_family}, while \labelcref{hat-is-smaller}, \labelcref{f_constant_on_equivalent_points} and \labelcref{f_increasing_step1} imply \labelcref{f_nondecreasing}. To alleviate notation, we shall denote $C(x) = \{g \leq g(x)\}$ for every $x \in \R^d$.

Let us show \labelcref{f_increasing_step1}. Consider $x,y\in \R^d$ such that $g(x) < g(y)$. Take a point $y'$ competitor in the definition of $\hat g(y)$, so such that $g(y') = g(y)$ and consider a sequence $(y'_n)$ converging to $y'$. We know that $y' \not\in C(x)$ otherwise we would have $g(y) = g(y') \leq g(x)$: a contradiction of $g(x) < g(y)$.  Since $C(x)$ is closed, $y'_n \not \in C(x)$ for $n$ large enough, thus $g(y'_n) > g(x)$, and taking the inferior limit yields:
\[\liminf_{n\to+\infty} g(y'_n) \geq g(x).\]
This holds for any $y'_n \to y'$ such that $g(y') = g(y)$, hence we get $\hat g(y) \geq g(x)$, and \labelcref{f_increasing_step1} is proved.

Let us show \labelcref{f_increasing} by contradiction, assuming that there exists $x,y$ such that $\hat g(x) \leq \hat g(y)$ but $g(x) > g(y)$. By \labelcref{hat-is-smaller} and \labelcref{f_increasing_step1}, we have
\begin{equation}\label{f_increasing_step3}
    \hat g(y) \leq g(y) \leq \hat g(x) \implies \hat g(x) = \hat g(y) = g(y) < g(x).
\end{equation}
We distinguish two cases.
\paragraph{Case 1: there exists $z$ such that $g(y) < g(z) < g(x)$.} In that case $g(z) \leq \hat g(x)$ by \labelcref{f_increasing_step1}, hence
\[\hat g(y) = g(y) < g(z) \leq \hat g(x),\]
a contradiction with $\hat g(y) = \hat g(x)$ in \labelcref{f_increasing_step3}.

\paragraph{Case 2: $\forall z, g(y) \leq g(z) \leq g(x) \implies g(z) = g(y) \text{ or } g(z) = g(x)$.} Consider $x'_n \to x'$ such that $g(x') = g(x)$ as candidate in the definition of $\hat g(x)$. We may assume without loss of generality that $g(x'_n) \leq g(x')$ (otherwise we replace $x'_n$ with $x'$ for all indices $n$ such that $g(x'_n) > g(x')$). Since $g(y) < g(x) = g(x')$, necessarily $x' \not \in C(y)$ thus $x'_n \not\in C(y)$ for $n$ large enough because $C(y)$ is closed, which yields
\[g(x) = g(x') \geq g(x'_n) > g(y).\]
In the present case, it implies that $g(x'_n)$ is equal either to $g(y)$ or to $g(x)$, thus necessarily to $g(x)$: $g(x'_n) = g(x)$ for every $n$ large enough. This is true for every candidate sequence $(x'_n)$ in the definition of $\hat g(x)$ thus $\hat g(x)=g(x)$: a contradiction with \labelcref{f_increasing_step3}.

In both cases, we have found a contradiction, hence \labelcref{f_increasing} holds true, which concludes the proof of \labelcref{prop_function_from_set_family}.

Finally, we show that $\hat g$ is lower semi-continuous. Let $(x_k)_{k\in\N}$ be a sequence converging to some $x$, and up to taking a subsequence, assume that $\liminf_{k\to+\infty} \hat g(x_k) = \lim_{k\to+\infty} \hat g(x_k)$. We want to show that $\hat g(x) \leq \lim_{k\to+\infty} \hat g(x_k)$. Note that, thanks to \labelcref{prop_function_from_set_family}, $\hat g(x) \leq \hat g(x_k)$ whenever $g(x) \leq g(x_k)$, thus we may without loss of generality remove all the corresponding indices $k$ from the sequence $(x_k)_{k\in\N}$ (assuming there remains infinitely many), so that for every $k$,
\begin{equation}\label{hat_g_strict_for_lsc}
    \hat g(x_k) < \hat g(x), \quad\text{or equivalently,} \quad g(x_k) < g(x).
\end{equation}
Moreover, one can extract a subsequence such that the real sequence $(\hat g(x_k))_{k\in\N}$ is monotone, and $(g(x_k))_{k\in\N}$ has the same monotonicity. Let us justify that $(\hat g(x_k))_{k\in \N}$ must be nondecreasing and eventually increasing. If not, it would eventually be nonincreasing. Thus for large $k$, $g(x_k) \leq g(x_N)$ for some fixed $N$, i.e. $x_k\in C(x_N)$, and since $C(x_N)$ is closed, $x=\lim_{k\to+\infty} x_k \in C(x_N)$, which means that $g(x) \leq g(x_N)$: a contradiction with \labelcref{hat_g_strict_for_lsc}. Thus, up to subsequence $(\hat g(x_k))_{k\in\N}$ is increasing. For every $k$, since $g(x_k) < g(x_{k+1})$, we know by \labelcref{f_increasing_step1} that $g(x_k) \leq \hat g(x_{k+1})$, thus
\[\liminf_{k\to +\infty} \hat g(x_k) = \liminf_{k\to+\infty} \hat g(x_{k+1}) \geq \liminf_{k\to+\infty} g(x_k) \geq \liminf_{k\to+\infty} \bar g(x_k) \geq \bar g(x) \geq \hat g(x),\]
where we have used \labelcref{f_characterization_lsc_envelope} in the last inequality.
\end{proof}

\begin{prp}[$\mathsf{F}_d \simeq \mathsf{P}_d \simeq \mathsf{C}_d$]\label{equivalence_potentials_relations_families}
The functions $\Phi : \mathsf{F}_d \to \mathsf{P}_d$ and $\Psi : \mathsf{P}_d \to \mathsf{C}_d$ defined in \labelcref{def_bijections} are bijective.

Moreover, for every $\mathcal C = (C(x))_{x\in \R^d} \in \mathsf{C}_d$,
\begin{equation}\label{Psi_inverse}
    \Psi^{-1}(\mathcal C) \text{ is the relation $\mathalpha{\preceq}$ defined by } (w\preceq x \iff w \in C(x)),
\end{equation}
and
\begin{equation}\label{PsiPhi_inverse}
    (\Psi \circ \Phi)^{-1}(\mathcal C) = \hat g/\mathalpha{\sim},
\end{equation}
where $\hat g$ is, for instance, the sublsc envelope (defined in \labelcref{def_function_from_set_family} of \Cref{sublsc_envelope}) of the real-valued quasi-convex function $g$ given by
\begin{equation}\label{pre_def_function_from_set_family}
g(x) = \dim C(x)  + \frac 1{(2\pi)^{\dim C(x)/2}} \int_{C(x)} e^{-\abs{w}^2/2} \dd \hdm^{\dim C(x)}(w),
\end{equation}
for every $x\in \R^d$. Here, $\hdm^k$ denotes the $k$-dimensional Hausdorff measure.
\end{prp}
\begin{rmk}
We shall see that the function $g$ defined in \labelcref{pre_def_function_from_set_family} satisfies $C(x) = \{ g \leq g(x)\}$ for every $x\in \R^d$, which implies that $g$ is quasi-convex, but it is not necessarily lsc. Indeed, although all sub-level sets of the form $\{ g \leq g(x)\}$ are closed, it may happen that for some levels $\ell \not \in g(\R^d)$, $\{g\leq \ell\}$ is not closed. Replacing $g$ with its lsc envelope $\bar g$ would not work since in general $\{ g \leq g(x)\} \neq \{ \bar g \leq \bar g(x)\}$. This is why we introduced the sub-level-stable lsc envelope $\hat g$ in \Cref{sublsc_envelope}.
\end{rmk}
\begin{proof}
First of all, it is straightforward to check that $\Psi$ is bijective with inverse given by \labelcref{Psi_inverse}, thus it suffices to show that $\Theta \coloneqq \Psi \circ \Phi$ is bijective with inverse given by \labelcref{PsiPhi_inverse}.

Let $\mathcal C = (C(x))_{x\in \R^d} \in \mathsf{C}_d$ and $\mathalpha{\preceq} \coloneqq \Psi^{-1}(\mathcal C) \in \mathsf{P}_d$ the associated preference relation. Let us show that the function $g$ defined in \labelcref{pre_def_function_from_set_family} satisfies
\begin{equation}\label{g_sub-level_sets}
    C(x) = \{g\leq g(x)\}:=C_g(x)\quad(\forall x\in \R^d).
\end{equation}
This will follow from the fact that $g$ is strictly increasing with respect to the family $\mathcal C$, in the sense that $g(x') \leq g(x)$ whenever $C(x') \subseteq C(x)$ (this part is obvious from the expression of $g$) and that $g(x') = g(x) \implies C(x') = C(x)$.

Let us prove the latter implication. Assume that $g(x') = g(x)$, and since $\mathcal C$ is totally ordered we may assume for example that $C(x') \subseteq C(x)$. If $\dim C(x') < \dim C(x)$ then
\begin{align*}
    g(x') &= \dim C(x')  + \frac 1{(2\pi)^{\dim C(x')/2}} \int_{C(x')} e^{-\abs{w}^2/2} \dd \hdm^{\dim C(x')}(w)\\
    &\leq \dim C(x') + 1\\
    &< \dim C(x) + \frac 1{(2\pi)^{\dim C(x)/2}} \int_{C(x)} e^{-\abs{w}^2/2} \dd \hdm^{\dim C(x)}(w) = g(x),
\end{align*}
which is a contradiction. Thus $\dim C(x') = \dim C(x) = k$, and $C(x), C(x')$ are subsets of a $k$-dimensional affine subspace $V_k$. If $C(x') \subsetneq C(x)$, since both $C(x),C(x')$ are closed convex sets and closed convex sets are the closure of their relative interior, $\overline{\relint C(x)} \setminus C(x') \neq \emptyset$, which implies that $\relint (C(x) \setminus C(x')) \neq \emptyset$, thus we can find a small ball $B_{V_k}(y,r)$ in $V_k$ with $r>0$ such that $B_{V_k}(y,r) \subseteq C(x)\setminus C(x')$. As a consequence
\begin{align*}
    g(x) &= k + \int_{C(x)} e^{-\abs{w}^2/2} \dd \hdm^k(w)\\
    &\geq k + \int_{C(x')} e^{-\abs{w}^2/2} \dd \hdm^k(w) + \int_{B_{V_k}(y,r)} e^{-\abs{w}^2/2} \dd \hdm^k(w) > g(x'),
\end{align*}
which contradicts $g(x)=g(x')$, thus $C(x') = C(x)$. We have thus proven that $g$ is strictly increasing with respect to the family $\mathcal C$.

Let us conclude the proof of \labelcref{g_sub-level_sets}. If $x' \in C(x)$ then by definition of $\mathsf{C}_d$ we have $C(x') \subseteq C(x)$, thus $g(x') \leq g(x)$. Conversely, if $x' \not\in C(x)$ then $C(x') \not\subseteq C(x)$ and necessarily (by definition of $\mathsf{C}_d$ again) $C(x) \subseteq C(x')$ thus $g(x) \leq g(x')$. Since $C(x')\neq C(x)$ we cannot have $g(x) = g(x')$, thus $g(x) < g(x')$ and $x'\not\in \{g\leq g(x)\}$. We conclude that $C(x) = \{g\leq g(x)\}$. This shows in particular that $g$ is quasi-convex.

By \Cref{sublsc_envelope}, $\hat g$ is lower semi-continuous and still satisfies
\begin{equation}\label{hat_g_sub-level_sets}
    C(x) = \{\hat g\leq \hat g(x)\}\quad(\forall x\in \R^d),
\end{equation}
thus $\hat g$ is quasi-convex, $\hat g/\mathalpha{\sim} \in \mathsf{F}_d$ and $\Theta(\hat g/\mathalpha{\sim}) = \mathcal C$. In particular $\Theta$ is surjective. To conclude that $\Theta$ is bijective with inverse given by \labelcref{PsiPhi_inverse}, it suffices to show that $\Theta$ is injective, thus to show that if two quasi-convex and lsc functions $g_1,g_2$ have the same (spatially-indexed) sub-level families, i.e.
\[C_1(x) \coloneqq \{g_1 \leq g_1(x)\} = \{g_2 \leq g_2(x)\} \eqqcolon C_2(x)\quad(\forall x\in \R^d),\]
then $g_1 \sim g_2$. Note that they must also have the same (spatially-indexed) level families:
\[\{g_1=g_1(x)\} = C_1(x) \setminus \bigcup_{y : C_1(y) \subsetneq C_1(x)} C_1(y) = C_2(x) \setminus \bigcup_{y : C_2(y) \subsetneq C_2(x)} C_2(y) = \{g_2 = g_2(x)\}.\]
As a consequence, it is possible to define a function $\phi : g_1(\R^d) \to g_2(\R^d)$ by setting $\phi(g_1(x)) = g_2(x)$ for every $x\in \R^d$. Note that if $g_1(x') < g_1(x)$ then $C_2(x') = C_1(x') \subsetneq C_1(x) = C_2(x)$ which implies that $g_2(x')< g_2(x)$, so that $\phi$ is strictly increasing and $g_1\sim g_2$.

\end{proof}

We get the following immediate corollary, which reduces \Cref{main_problem} to finding a suitable totally ordered family of convex sets.

\begin{cor}\label{existence_potentials}
\begin{enumerate}[(1)]
    \item If $\mathcal C = (C(x))_{x\in \R^d}$ is a family of convex closed sets belonging to $\mathsf{C}_d$\modif{, i.e. which are totally ordered, satisfy} $x\in C(x)$ and ($w\in C(x) \implies C(w)\subseteq C(x)$) for every $x,w\in \R^d$, then there exists a quasi-convex lsc function $f$ such that
\[ C_f (x):=\{f\leq f(x)\} = C(x)\quad(\forall x\in \R^d).\]
\item If $g$ is a quasi-convex function, there exists a quasi-convex lsc function $f$ such that
\[C_f (x):= \{f\leq f(x)\} = \overline{\{g\leq g(x)\}}\quad(\forall x\in \R^d).\]
\end{enumerate}
\end{cor}
\begin{proof}
The first item is merely a rephrasing of the surjectivity of the map $\Psi \circ \Phi : \mathsf{P}_d \to \mathsf{C}_d$ established in \Cref{equivalence_potentials_relations_families}, and the second item is a consequence of the first one, applied to the family $\mathcal C \coloneqq (\overline{\{g\leq g(x)\}})_{x\in \R^d} \in \mathsf{C}_d$.
\end{proof}

\subsection{From cyclically quasi-monotone maps to preference relations and set families}

Given a multi-valued map $F :  \R^d \multimap \R^d$ (with full domain) which is cyclically quasi-monotone, we start by defining a binary relation $\prec_F$ as follows:
\begin{align}
    x \prec_F y &\iff \text{there is an $F$-ascending path from $x$ to $y$}\label{def_strict_pre-order}\\
    \shortintertext{i.e.}
     x \prec_F y &\iff \begin{multlined}[t]
     \text{there exists } (x_0,p_0), \ldots, (x_{N-1},p_{N-1}) \in F, x_N \in \R^d \text{ s.t.}\\
     x_0=x, x_N = y \text{ and } \forall i< N, (x_{i+1}-x_i)\cdot p_i > 0.
     \end{multlined}\notag
\end{align}
From this, we define the binary relation $\preceq_F$ that will be our candidate preference relation, as
\begin{align}
   x \preceq_F y &\iff   (\forall w\in \R^d, w \prec_F x \implies w\prec_F y)\label{def_pre-order}\\
 \shortintertext{or equivalently}
    x \preceq_F y &\iff   \{w : w \prec_F x\} \subseteq \{w : w \prec_F y\}.\label{charac_pre-order}
\end{align}

Recalling the vocabulary on binary relations given in \Cref{vocabulary_binary_relations}, let us state the first properties satisfied by $\prec_F$ and $\preceq_F$.

\begin{prp}\label{transitivity_properties}
Let $F : \R^d \multimap \R^d$ be a cyclically quasi-monotone multi-map.
\begin{enumerate}[(i)]
\item\label{item_strict_pre-order} The relation $\prec_F$ defined in \labelcref{def_strict_pre-order} is a \emph{strict pre-order}, i.e. an irreflexive and transitive relation. As such, it is also asymmetric.
\item\label{item_pre-order} The relation $\preceq_F$ defined in \labelcref{def_pre-order} is a pre-order, i.e. a reflexive and transitive relation.
\item\label{item_strict_implies_large} The strict relation is stronger than the large one: for every $x,y\in \R^d$,
\begin{equation}\label{strict_implies_large}
    x \prec_F y \implies x\preceq_F y.
\end{equation}
\item\label{item_mixed_transitivity} The following mixed transitivity relation holds: for every $x,y,z\in \R^d$,
\begin{equation}\label{strlarstr}
    x\prec_F y \text{ and } y\preceq_F z \implies x \prec_F z.
\end{equation}
\end{enumerate}
\end{prp}
\begin{proof}
The mixed transitivity property \labelcref{item_mixed_transitivity} holds by definition of $\preceq_F$ as given in \labelcref{def_pre-order}.

The transitivity of the strict relation $\prec_F$ is an immediate consequence of the fact that concatenating an $F$-ascending path from $x$ to $y$ with an $F$-ascending path from $y$ to $z$ yields an $F$-ascending path from $x$ to $z$. Besides, there exists $x\in \R^d$ such that $x\prec_F x$ if and only if there exists an $F$-ascending cycle, thus the cyclical quasi-monotonicity of $F$ is equivalent to the irreflexivity of $\prec_F$, and \labelcref{item_strict_pre-order} holds true.

As for \labelcref{item_pre-order}, it is a consequence of the characterization of $x \preceq_F y$, given in \labelcref{charac_pre-order}, by the inclusion $\{\cdot \prec_F x\} \subseteq \{\cdot \prec_F y\}$. Thus, the reflexivity and transitivity of $\preceq_F$ follow from the reflexivity and transitivity of the inclusion relation $\subseteq$.

Finally, take $x \prec_F y$ and $y \preceq_F z$. The latter means that $\{\cdot \prec_F y \} \subseteq \{\cdot \prec_F z\}$, thus $x\prec_F z$, and \labelcref{item_strict_implies_large} is proved.
\end{proof}

\begin{rmk}[On the definition of the pre-order]
There is a very natural way to define a strict pre-order from a pre-order $\preceq$, by setting $x \prec y$ if $x\preceq y$ but $y\not\preceq x$. An alternative would be to consider $x \prec_{\mathalpha{\neq}} y$ if $x\preceq y$ but $x\neq y$. Unfortunately, this would not work for a pre-order which is not an order (i.e. not antisymmetric), since $\prec_{\mathalpha{\neq}}$ could fail to be transitive. It seems $\prec$ is the only natural choice.

However, doing the converse is more complicated, since one can think of many natural ways to define a pre-order $\preceq$ from a strict pre-order $\prec$. This is due to the fact that \modif{there} are many possibilities to define equivalent points, i.e. points such that $x\preceq y$ and $y\preceq x$. One possibility is to take $x\preceq' y \iff y \not\prec x$, which is a preference relation only when the induced incomparability relation\footnote{The incomparability relation $\sim$ is defined by $x\sim y$ iff neither $x \prec y$ nor $y \prec x$.} is transitive, in which case $\prec$ is said to be a \emph{strict partial order}. Moreover the map $\mathalpha{\prec} \mapsto \mathalpha{\preceq'}$ is a bijection between strict partial orders and preference relations, and is the inverse of the map $\mathalpha{\preceq} \mapsto \mathalpha{\prec}$ defined in the previous paragraph. When $\prec$ is not a strict partial order, then $\preceq'$ may fail to be transitive, although it is reflexive and total. We decided to trade totality for transitivity: defining, as we did, $x \preceq y$ if $\{\cdot \prec x\} \subseteq \{ \cdot \prec y\}$, yields a reflexive and transitive relation, i.e. a pre-order, that may lack totality to be a preference relation. We found it more convenient to study a relation that is transitive and look for totality rather \modif{than} to study a relation that is total but lacks transitivity. Finally, let us note that both relations $\preceq$ and $\preceq'$ are equal when $\prec$ is a strict partial order.
\end{rmk}

We continue with an easy but useful characterization of $\preceq_F$.
\begin{lem}[One-point characterization of $\preceq_F$]\label{pre-order_one-point_charac}
Given a cyclically quasi-monotone multi-map $F : \R^d \multimap \R^d$. For every $x,y\in \R^d$, the pre-order $\mathalpha{\preceq}_F$ defined in \labelcref{def_pre-order} satisfies:
\begin{equation}\label{pref_one_point_charac}
    x \preceq_F y \iff (\forall (w,p_w)\in F, \quad (x-w)\cdot p_w >0\implies w \prec_F y).
\end{equation}
\end{lem}
\begin{proof}
The direct \modif{implication} is straightforward, because if $(w,p_w) \in F$ is such that $(x-w)\cdot p_w > 0$ then $(w,x)$ is an $F$-ascending path, so that $w \prec_F x$, and the mixed transitivity relation \labelcref{strlarstr} gives $w \prec_F y$.

For the converse implication, we assume the right-hand side of \labelcref{pref_one_point_charac} is satisfied and we consider $x, y \in \R^d$. For every $w \prec_F x$, we take any $F$-ascending path $w = x_0, \ldots x_N = x$, together with $p_0\in F(x_0),\ldots, p_{N-1} \in F(x_{N-1})$ such that $(x_{i+1}-x_i) \cdot p_i > 0$ for every $i < N$. \modif{By the right-hand side of \labelcref{pref_one_point_charac}}, $(x-x_{N-1}) \cdot p_{N-1} > 0$ implies that $x_{N-1} \prec_F \modif{y}$. Thus we have $w = x_0 \prec_F x_1, x_1 \prec_F x_2, \ldots, x_{N-1} \prec_F \modif{y}$, and by transitivity of $\prec_F$ (\Cref{transitivity_properties} \labelcref{item_strict_pre-order}) we get $w \prec_F y$. We have shown $w \prec_F x \implies w \prec_F y$, which exactly means that $x \preceq_F y$.
\end{proof}

We are now ready to show that $\preceq_F$ is quasi-convex lsc and relate the normal cones of the convex sets belonging to the family
\begin{equation}\label{def_convex_multi-map}
    \C^F \coloneqq (C^F(x))_{x\in \R^d} \coloneqq \Psi(\mathalpha{\preceq}_F)
\end{equation}
to the original multi-map $F$.
In other words, $y \in C^F (x)$ if $y \preceq_F x$, which means that for all $z \prec_F y$, it holds $z \prec_F x$.

\begin{prp}\label{pre-order_properties}
Let $F : \R^d \multimap \R^d$ be a cyclically quasi-monotone multi-map and $\preceq_F$ the pre-order defined in \labelcref{def_pre-order}. The following holds:
\begin{enumerate}[(i)]
\item\label{pre-order_quasi-convex} the pre-order $\preceq_F$ is quasi-convex and lsc;
\item\label{pre-order_normal_cone} for every $x\in \R^d$, $F(x) \subseteq N_{C^F(x)}(x)$, $C^F(x)$ being defined in \labelcref{def_convex_multi-map}.
\end{enumerate}
\end{prp}
\begin{proof}
Let us take $x \in \R^d$. By \Cref{pre-order_one-point_charac}, a point $x'$ satisfies $x'\preceq_F x$ if and only if
\[\forall w \not \prec_F x, \forall p_w \in F(w),\quad (x'-w) \cdot p_w \leq 0,\]
which means that
\begin{equation}\label{one_step_description_C}
    C^F(x) = \bigcap_{(w,p_w)\in F : w \not \prec_F x} \{(\cdot - w) \cdot p_w \leq 0\},
\end{equation}
thus $C^F(x)$ is closed and convex as an intersection of closed half-spaces. We have proved the first item.

For the second item, consider again a point $x \in \R^d$ and take $p_x \in F(x)$. For every $y \in C^F(x)$, $y\preceq_F x$ thus \modif{by \modif{\labelcref{item_mixed_transitivity}} of \Cref{transitivity_properties}}  $x\not \prec_F y$, therefore $(y-x)\cdot p_x$ is necessarily smaller or equal than $0$. This is true for every $y \in C^F(x)$ thus $p_x \in N_{C^F(x)}(x)$ by definition of the normal cone.
\end{proof}

Whenever $\preceq_F$ is total \Cref{main_problem} can be solved, as stated in the following.

\begin{prp}[Conditional resolution of \Cref{main_problem}]\label{cond-resol}
Let $F : \R^d \multimap \R^d$ be a cyclically quasi-monotone multi-map. If the pre-order $\preceq_F$ defined in \labelcref{def_pre-order} is total, then it is a quasi-convex and lsc preference relation, and there exists a lsc quasi-convex function $f : \R^d \to \R$ such that
\[\forall x \in \R^d, \quad F(x) \subseteq \nop f(x).\]
\end{prp}
\begin{proof}
By \Cref{transitivity_properties}, if $\preceq_F$ is total, then it is a preference relation, and by \Cref{pre-order_properties} \labelcref{pre-order_quasi-convex} it is quasi-convex and lsc, thus $\mathalpha{\preceq}_F \in \mathsf P_d$. Since the family $\C^F$ induced by $\preceq_F$ belongs to $\mathsf C_d$ (see \Cref{equivalence_potentials_relations_families}), we may apply \Cref{existence_potentials} to find a quasi-convex lsc potential $f : \R^d \to \R$ such that $C_f(x) = C^F(x)$ for every $x\in \R^d$. By \Cref{pre-order_properties} \labelcref{pre-order_normal_cone}, it yields the following:
\[\forall x\in \R^d, \quad F(x) \subseteq N_{C^F(x)}(x) = N_{C_f(x)}(x) = \nop f(x).\]
\end{proof}
\modif{We will show in \Cref{ND-lipschitz} a case when \labelcref{def_pre-order} is total and the previous proposition may then be applied.}

\section{The one-dimensional case}\label{sec:1d}
In the $1$-d case, we can give a positive solution to \Cref{main_problem}. Let $F: \R \multimap \R \cup \{+\infty\}$ be a cyclically quasi-monotone map (with full domain). In this section $\preceq$ will always refer to $\preceq_F$ and $C_*$ to the following convex set:
\[C_* \coloneqq \bigcap_{x \in \R} C^F(x) =\{w \in \R \ : \ w \preceq x \ (\forall x)\}.\]
Note that if $x \in C_*$ then $C^F(x) = C_*$, and moreover if $x\in \overset{\circ}{C_*}$ then $F(x)=\{0\}$ by \cref{pre-order_normal_cone} of \Cref{pre-order_properties}.

As already noticed in \cite{daniilidisSubdifferentialsQuasiconvexPseudoconvex1999}, cyclical quasi-monotonicity coincides with mere quasi-monotonicity in $1$-d. It is equivalent to saying that in the definition of $\prec$ we need only to consider paths made of two points.

\begin{lem}[Consequence of {\cite[Proposition~3.3]{daniilidisSubdifferentialsQuasiconvexPseudoconvex1999}}]\label{strict_1d}
For every $x,y \in \R$, $x\prec y$ if and only if there exists $p\in F(x)$ such that $p \cdot (y-x) > 0$.
\end{lem}
We provide a proof for sake of completeness.
\begin{proof}
Assume $x\prec y$, so that there exists a path $x = x_0, \ldots, x_N = y$ and reals $p_i \in F(x_i)$ such that $p_i\cdot (x_{i+1}-x_i) > 0$ for every $0\leq i < N$. Assume without loss of generality that $p_0 > 0$. If $p_i > 0$ for every $i$ then $x_0 < x_N$ thus $p_0\cdot (x_N-x_0) > 0$, while if that is not the case taking the minimal $j > 0$ such that $p_j < 0$ yields $\min \{p_0\cdot (x_j -x_0), p_j\cdot (x_0-x_j)\} > 0$ which violates the cyclical quasi-monotonicity of $F$. Thus we are in the first case and $p_0 \cdot(y-x) > 0$.
\end{proof}

\begin{lem}\label{lem_1d}
Assume that $F \not\equiv \{0\}$. Setting
\[\alpha \coloneqq \sup \{x : \exists p_x \in F(x), p_x < 0\} \quad\text{and}\quad \beta \coloneqq \inf \{x : \exists p_x \in F(x), p_x > 0\},\]
with values in $[-\infty,+\infty]$ with the usual convention, it holds
\begin{enumerate}[(i)]
\item $\alpha \leq \beta$;\label{alpha_leq_beta}
\item $x<\alpha \implies F(x) \subseteq \R_-$;\label{left_negative}
\item $ x > \beta \implies F(x) \subseteq \R_+$; \label{right_positive}
\item $\alpha < x < \beta \implies F(x) = \{0\}$; \label{middle_zero}
\item $C_* = \{ x \in \R : \alpha \leq x \leq \beta\}$. \label{middle_is_minimal}
\end{enumerate}
\end{lem}
\begin{proof}
    Let $x$ such that $F(x)$ contains some $p_x > 0$. Take $y > x$. If there exists $p_y \in F(y)$ such that $p_y < 0$ then $\min \{(y-x) \cdot p_x, (x-y) \cdot p_y\} > 0$, which contradicts the quasi-monotonicity of $F$, thus $F(y) \subseteq \R_+$ for every $y > x$, and \labelcref{right_positive} holds true. \labelcref{left_negative} holds true by similar arguments, and \labelcref{alpha_leq_beta} and \labelcref{middle_zero} then follow from the definition of the infimum and supremum.

    If $x > \beta$ then there exists $y \in (\beta, x)$ and $p_y\in F(y)$ such that $p_y > 0$. Then $y \prec x$ so that $x \not \in C_*$. A similar reasoning holds when $x < \alpha$. Therefore, $C_* \subseteq \{x \in \R : \alpha \leq x \leq \beta\}$. Conversely, take a real $x\in [\alpha,\beta]$. Since $x\leq \beta$ there is no $y < x$ with $p_y > 0$ and $p_y \in F(y)$ so $y \not \prec x$. Since $x\geq \alpha$ the same holds for $y > x$, therefore $\{\cdot \prec x\} = \emptyset$ so that $x\in C_*$.
    \end{proof}

\begin{thm}
If $F : \R \multimap \R\cup\{+\infty\}$ \modif{is} a cyclically quasi-monotone multi-map, then there exists a lower semi-continuous quasi-convex function $f : \modif{\R} \to \R$ such that $F\subseteq \nop f$.
\end{thm}
\begin{proof}
    We distinguish two cases, depending on $C_*$.

    If $C_* = \emptyset$, then following the notation of \Cref{lem_1d}, either $\alpha = \beta = +\infty$ or $\alpha = \beta = -\infty$. In the first case, we take $f(x) = -x$ while in the second case we take $f(x) = x$. By \Cref{lem_1d} $F(x) \subseteq \nop f(x)$ for every $x$ and the problem is solved.

    If $C_* \neq \emptyset$, we consider $f(x) = \dist(x,C_*)$. When $F \equiv \{0\}$, we have $\alpha = -\infty, \beta = +\infty$, $C_* = \R$ and $f \equiv 0$. When $\alpha = \beta = x_*$, $\nop f(x_*) = \R$ and by \Cref{lem_1d}, we get that $f$ is a potential for $F$ in all cases.

\end{proof}

\section{The regular and non-vanishing case}\label{sec:nd}

In this section, we will give \modif{a} positive answer to \Cref{main_problem} in the case of a single-valued and continuous operator $F : \R^d \to \R^d$ which is cyclically quasi-monotone and satisfies the following non-vanishing condition\footnote{In this case, $F$ actually satisfies a stronger condition called \emph{cyclic pseudo-monotonicity}.}:
\begin{enumerate}[(NV)]
\item\label{hyp_nonzero} $F(x) \neq 0$ for every $x\in\R^d$.
\end{enumerate}
Single-valuedness of $F$ will be a standing assumption in this section. \modif{Moreover, also by the non-vanishing condition, replacing $F$ by $\frac{F}{\abs{F}}$ will not affect the results of this section.} We start with a general lemma providing information on the interior, boundary, and boundedness of the convex sets belonging to the family $\mathcal C^F$ generated by $F$ (see \labelcref{def_convex_multi-map}).

\begin{lem}\label{convex_prop_cont_vanish}
Let $F$ be a continuous and cyclically quasi-monotone field satisfying \cref{hyp_nonzero}. For every $x\in \R^d$, we have
\begin{enumerate}[(i)]
\item\label{convex_interior} $\emptyset \neq \{z : z \prec x\} \subseteq \intr C^F(x)$ and $C^F(x) \neq \R^d$,
\item\label{convex_normal} for every $y \in \bd C^F(x)$, $N_{C^F(x)}(y) = \spn_+ F(y)$, and \modif{$C^F(x)$ has a $\mathscr C^1$ boundary},\footnote{\modif{We have denoted $\spn_+(u) = \{tu : t\in \R_+\}$.}}
\item\label{convex_unbounded} $C^F(x)$ is unbounded.
\end{enumerate}
\end{lem}

\begin{rmk}
\modif{%
    In particular, if $F$ is Lipschitz then the boundary of $C^F(x)$ will be of class $C^{1,1}$.
}
\end{rmk}

\begin{proof}
Let us consider $z_t \coloneqq x - t F(x)$, where $t \in (0,1)$. For $t$ small enough, by continuity of $F$, we have
\[(x-z_t) \cdot F(z_t) = t F(x) \cdot F(z_t) > 0,\]
thus $z_t \prec x$ and $\{z : z \prec x\}  \neq \emptyset$. Now take any $z \prec x$: there is an ascending path $z = x_0, x_1, \ldots, x_n = x$, i.e. such that $F(x_i)\cdot (x_{i+1}-x_i) > 0$ for every $i\in\{0,\ldots,n-1\}$. By continuity of $F$, there exists a small $\eps > 0$ such that for every $z' \in B_\eps(z)$,
\[(x_1 - z') \cdot F(z') > 0,\]
thus the path $z', x_1, \ldots, x_n = x$ is an ascending path and $z' \prec x$. Finally, since for $t>0$, $x \prec (x+tF(x))$, $x + tF(x) \not\in C^F(x)$ thus $C^F(x) \neq \R^d$. We have proved \labelcref{convex_interior}.

To prove \labelcref{convex_normal}, notice first that, by the discussion above, $x\in \bd C^F(x)$ which is, then, not empty. Then take a point $y\in \bd C^F(x)$ and consider a sequence $y_n \not\in C^F(x)$ converging to $y$ as $n\to+\infty$. Necessarily,
\[\forall z \in C^F(x), \quad (z-y_n) \cdot F(y_n) \leq 0\]
otherwise we would have $y_n \prec z$ and since $z \preceq x$ then $y_n \prec x$ which contradicts $y_n \not\in C^F(x)$. We take the limit $n \to +\infty$ so as to get
\begin{equation}\label{convex_negative_scalar_product}
(z-y)\cdot F(y) \leq 0.
\end{equation}
Therefore, $F(y) \in N_{C^F(x)}(y)$. Let us prove by contradiction that $N_{C^F(x)}(y)$ is a half line at every $y\in \bd C^F(x)$, which will thus show the result. Let $y\in \bd C^F(x)$ be such that there exist two unit normal vectors $\nu,\nu' \in N_{C^F(x)}(y)$ which generate extremal rays. Since $C^F(x)$ has non-empty interior, the boundary of $C^F(x)$ is locally the graph of a convex (or concave) function $g: \R^{d-1} \to \R$. By the regularity of convex functions, there exist two sequences of points $(y_n), (y_n') \subseteq \bd C^F(x)$, where there exist unique unit normal vectors $\nu_n, \nu'_n$ and such that $\nu_n \to \nu$, $\nu'_n \to \nu'$. In that case, by what we just proved and thanks to \labelcref{hyp_nonzero}, for every $n\in\N$, we have
\[\nu_n = \frac{F(y_n)}{\abs{F(y_n)}} \quad\text{and}\quad \nu'_n = \frac{F(y'_n)}{\abs{F(y'_n)}}\]
and by continuity of $F$,
\[\frac{F(y)}{\abs{F(y)}} = \lim_{n\to+\infty} \frac{F(y_n)}{\abs{F(y_n)}} = \nu \neq \nu' = \lim_{n\to+\infty} \frac{F(y'_n)}{\abs{F(y'_n)}} = \frac{F(y)}{\abs{F(y)}},\]
which is a contradiction. This shows that $F/\abs{F}$ \modif{is the} unit normal field at the boundary of $C^F(x)$ and since $F$ is \modif{of} class $\mathscr C^0$ then, again writing that $\bd C^F(x)$ is locally the graph of a convex function, we conclude that $C^F(x)$ is of class $\mathscr C^1$, concluding the proof of \labelcref{convex_normal}.

Let us show that $C^F(x)$ is unbounded. We argue by contradiction, assuming that it is bounded. In that case the set $C^F(x) = \{z : z\preceq_F x\}$ equipped with $\preceq_F$ is inductive: if a subset $V$ of $C^F(x)$ is totally ordered, the set $U \coloneqq \bigcap_{z\in \modif{V}} C^F(z)$ is nonempty as the intersection of totally ordered nonempty compact sets, and any $z\in U$ is a minorant of $V$. Zorn's lemma then guarantees the existence of a minimal element $z^*$ in $C^F(x)$. By \labelcref{convex_interior} the set $\{z : z\prec z^*\} \subseteq \modif{C^F(x)}$ is nonempty: a contradiction to the minimality of $z^*$.
\end{proof}

We proceed by showing that the convex sets in the family $\C^F$ always have a non-empty pairwise intersection.

\begin{lem}\label{convex_sets_intersect}
Assume that $F$ is a cyclically quasi-monotone and continuous field that satisfies \labelcref{hyp_nonzero}. For every $x,x' \in \R^d$, $\intr(C^F(x)) \cap \intr(C^F(x')) \neq \emptyset$.
\end{lem}
\begin{proof}
We start by showing that $C^F(x) \cap C^F(x') \neq \emptyset$ for every $x,x'\in \R^d$. We argue by contradiction, assuming that for some $x\neq x'$, $C^F(x) \cap C^F(x') = \emptyset$. We distinguish two cases.

\paragraph{Case 1: $d(C^F(x),C^F(x'))$ is attained.} Let $y \in C^F(x), y' \in C^F(x')$ such that
\[d(C^F(x),C^F(x')) = \abs{y-y'} > 0.\]
We know that $y' = p_{C^F(x')}(y)$ and $y = p_{C^F(x)}(y')$, and since by \labelcref{convex_normal} of \Cref{convex_prop_cont_vanish} $F$ is a normal field to $C^F(x)$ and $C^F(x')$, we have
\[y'-y \in N_{C^F(x)}(y) = \spn_+(F(y)) \quad\text{and}\quad y-y' \in N_{C^F(x')}(y') = \spn_+(F(y')).\]
As a consequence $(y'-y)\cdot F(y) > 0$ and $(y-y')\cdot F(y') > 0$: a contradiction \modif{to} the quasi-monotonicity of $F$.

\paragraph{Case 2: $d(C^F(x),C^F(x'))$ is not attained.} We shall use Ekeland's variational principle (see \cite{ekelandNonconvexMinimizationProblems1979}). Let us define $\Phi = \Phi_1 + \Phi_2$ where $\Phi_1,\Phi_2: \R^d \times \R^d \to [0,+\infty]$ are given by
\[\Phi_1(Y) = \abs{y-y'}, \qquad \Phi_2(Y) = \chi_{C^F(x)\times C^F(x')}(y,y'),\]
for every $Y=(y,y') \in \R^d \times \R^d$. The function $\Phi$ is a proper and lower semi-continuous convex \modif{function} on $\R^d \times \R^d$ and it is bounded from below (by $0$). For every $\eps > 0$ we may apply Ekeland's variational principle to obtain a point $Y_* = (y_*,y_*')$ which satisfies in particular:
\begin{align*}
    \forall\, Y\neq Y_*, \quad \Phi(Y) + \eps \abs{Y-Y_*} > \Phi(Y_*).
\end{align*}
In particular, $Y_*$ is a minimizer of the map $\Phi + \Phi_3$ where $\Phi_3 :  Y \mapsto \eps\abs{Y-Y_*}$. The subgradients can be easily computed as follows:
\begin{align*}
   \partial \Phi_1(y,y') &=\left\{ \left(\frac{y-y'}{\abs{y-y'}},-\frac{y-y'}{\abs{y-y'}}\right)\right\}\qquad (\forall y\neq y'),\\
 \partial \chi_{C^F(x)\times C^F(x')}(y,y') &= \partial \chi_{C^F(x)}(y) \times \partial \chi_{C^F(x')}(y') \qquad (\forall y,y'),\\
\partial \Phi_3(Y_*) &= \eps \modif{\uball^{2d}}.
\end{align*}
Here, $\modif{\uball^{2d}}$ denotes the closed unit ball in $\R^{2d}$. Besides,
\[\bigcap_{i\in\{1,2,3\}}\intr(\dom \Phi_i) = \intr(\dom\Phi_2) = \intr(C^F(x)) \times \intr (C^F(x')) \neq \emptyset,\]
thus the sub-gradient of $\Phi_1+\Phi_2+\Phi_3$ is the sum of the sub-gradients of the three functions. Since $C^F(x) \cap C^F(x') = \emptyset$, $y_* \neq y_*'$, \modif{and} the first-order optimality condition yields $0 \in \partial (\Phi+\Phi_3)(Y_*)$, i.e. \modif{the} existence of $(\nu,\nu') \in \partial \chi_{C^F(x)}(y_*)\times \partial \chi_{C^F(x')}(y_*')$ and $(u,u') \in \uball^{2d}$ such that
\[ 0 = \begin{pmatrix}\displaystyle \frac{y_*-y_*'}{\abs{y_*-y_*'}} + \eps u + \nu\\ \displaystyle \frac{y_*'-y_*}{\abs{y_*'-y_*}} + \eps u' + \nu'
\end{pmatrix}.\]
As a consequence \modif{both vectors $\nu,\nu'$ have norm at most $1+\eps$, thus for $\eps$ small enough,}
\[ \nu \cdot \frac{y_*'-y_*}{\abs{y_*'-y_*}} = 1 \modif{+} \varepsilon u \cdot \nu \geq \frac 12 > 0\]
and similarly
\[ \nu' \cdot \frac{y_*-y_*'}{\abs{y_*-y_*'}} > 0.\]
Since $\nu \in \spn_+ F(y_*)$ and $\nu'\in\spn_+ F(y_*')$, we obtain $y_* \prec y'_* \prec y_*$: a contradiction to the quasi-monotonicity of $F$. We have thus shown that $C^F(x) \cap C^F(x') \neq \emptyset$ for every $x,x'\in \R^d$.

Now let us prove that they always intersect in their interior. Assume by contradiction that $\intr(C^F(x)) \cap \intr(C^F(x')) = \emptyset$ for some $x\neq x'$. This implies that $\intr(C^F(x)) \cap \bd C^F(x') = \intr(C^F(x')) \cap \bd C^F(x) = \emptyset$. Indeed if, for example $\intr(C^F(x)) \cap \bd C^F(x') \neq \emptyset$, since $C^F(x') = \overline{\intr(C^F(x'))}$ then $\intr(C^F(x)) \cap \intr( C^F(x')) \neq \emptyset$. Since $C^F(x) \cap C^F(x') \neq \emptyset$ we must have $\bd C^F(x) \cap \bd C^F(x') \neq \emptyset$. Take a point $y \in \bd C^F(x) \cap \bd C^F(x')$: since the convex sets intersect at $y$ but their interior have empty intersection, they are tangent at $y$, and because they are of class $\mathscr C^1$, by \Cref{convex_prop_cont_vanish}, they have opposite normal at $y$. By the same lemma the normal to both convex sets is positively spanned by $F(y)$, hence a contradiction.
\end{proof}

In order to show that the convex sets of the family $\C^F$ are totally ordered, we need a final perturbation lemma.

\begin{lem}\label{lemma_perturbation}
Let $F : \R^d \to \R^d$ be a cyclically quasi-monotone map satisfying \cref{hyp_nonzero} and assume that $F$ is locally Lipschitz continuous. Let $p\in \R^d$, $K$ a compact subset of $\bd \modif{C^F}(p)$ and $T > 0$. For every $\eps > 0$, there exists $\delta_0= \delta_0(\eps,F,K,T) > 0$ satisfying the following property: for every $0< \delta \leq \delta_0$ and every tuple of points $(x_0,\ldots,x_n) \subseteq K$ such that
\[n\delta \leq 2 T \qquad \text{and}\qquad \abs{x_{i+1}-x_i} = \delta \qquad (\forall i\in \{0,\ldots,n-1\}),\]
there exists a perturbed family $(\tilde x_0,\ldots, \tilde x_n) \subseteq \R^d$ such that
\begin{enumerate}[(i)]
\item\label{perturbed_initial} $\tilde x_0 = x_0$,
\item\label{perturbed_is_close} $\abs{\tilde x_i-x_i} \leq \eps$ for every $i\in\{0,\ldots,n\}$,
\item\label{perturbed_is_increasing} $(\tilde x_{i+1}-\tilde x_i) \cdot F(\tilde x_i) > 0$ for every $i\in \{0, \ldots,n-1\}$.
\end{enumerate}
\end{lem}
\begin{proof}
Let $\eps > 0$. First of all, let us note that thanks to \cref{hyp_nonzero}, up to replacing $F$ with $F/\abs{F}$, we may assume that $\abs{F(x)} = 1$ for every $x\in \R^d$.

We can write, locally, as in \Cref{convex_prop_cont_vanish} above, $\bd \modif{C^F}(p)$ as the graph of a function which is $C^{1,1}$, since in this lemma $F$ is locally Lipschitz. Then we obtain that for $\delta > 0$ small enough and some constant $A$, it holds:
\begin{equation*}
    \forall x,y \in K, \quad \abs{y-x} \leq \delta \implies d(y,\tanspace_x C^F(p)) \leq A \abs{y-x}^2.
\end{equation*}
Note that by \Cref{convex_prop_cont_vanish}, $\tanspace_x C^F(p) = x + F(x)^\perp$ thus this is equivalent to
\begin{equation}\label{angle_inequality}
    \forall x,y \in K, \quad \abs{y-x} \leq \delta \implies \abs{(y-x)\cdot F(x)} \leq A \abs{y-x}^2.
\end{equation}

Let us take a tuple of points $(x_0,\ldots, x_n) \subseteq K$ as in the statement:
\[\abs{x_{i+1}-x_i} =\delta \qquad (\forall i\in \{0,\ldots,n-1\}).\]
We define the perturbed tuple $(\tilde x_i)_{0\leq i \leq n}$ as follows:
\[\left\{\begin{aligned}
\tilde x_0 &= x_0,\\
\tilde x_{i} &= x_{i} + \eps_i F(x_{i-1}) \qquad (0< i \leq n),
\end{aligned}\right.\]
for a suitable tuple $(\eps_1,\ldots,\eps_n)$ of \modif{increasing} values in $(0,1)$ to be chosen later.

Let us estimate from below the scalar product $F(\tilde x_i) \cdot (\tilde x_{i+1}-\tilde x_i)$ for $i=0$ and $1<i < n$:
\begin{equation}\label{perturbed_sequence_lower_bound_0}
\begin{aligned}
    (\tilde x_1-\tilde x_0) \cdot F(\tilde x_0) &= (x_1 + \eps_1 F(x_0) - x_0) \cdot F(x_0)\\
    &= (x_1-x_0) \cdot F(x_0) + \eps_1 \geq -A\delta^2 + \eps_1
    \end{aligned}
\end{equation}
and
\begin{equation}\label{perturbed_sequence_lower_bound}
\begin{aligned}
    &\qquad (\tilde x_{i+1}-\tilde x_i) \cdot F(\tilde x_i)\\
    &= (x_{i+1}-x_i)\cdot F(\tilde x_i) + (\eps_{i+1} F(x_{i}) - \eps_i F(x_{i-1}))\cdot F(\tilde x_i)\\
    &= (x_{i+1}-x_i)\cdot F(\tilde x_i) + (\eps_{i+1}-\eps_i) F(x_{i})\cdot F(\tilde x_i) + \eps_i (F(x_{i})-F(x_{i-1}))\cdot F(\tilde x_i)\\
    &= \begin{multlined}[t]
    (x_{i+1}-x_i) \cdot F(x_i) + (x_{i+1}-x_i) \cdot (F(\tilde x_i)-F(x_i))\\
    + (\eps_{i+1}-\eps_i) F(x_{i})\cdot F(\tilde x_i)+ \eps_i (F(x_{i})-F(x_{i-1}))\cdot F(\tilde x_i).
    \end{multlined}
\end{aligned}
\end{equation}

We may lower bound the last four terms of \labelcref{perturbed_sequence_lower_bound} using \labelcref{angle_inequality} (for the first term), the fact that $F$ is Lipschitz on $K + B(0,1)$ with constant $L$, so as to obtain the following:
\begin{equation}\label{lower_bound_perturbed}
    (\tilde x_{i+1}-\tilde x_i) \cdot F(\tilde x_i) \geq -A\delta^2 - 2L\delta \eps_i + (\eps_{i+1}-\eps_i) (1-L\eps_i).
\end{equation}
To get a positive scalar product in \labelcref{perturbed_sequence_lower_bound_0} and \labelcref{perturbed_sequence_lower_bound}, we define $(\eps_i)_{1\leq i\leq n}$ as follows:
\[\left\{\begin{aligned}
\eps_1 &= \frac{2A}{L} \delta,\\
\eps_{i+1} &= \left(1 + \frac{3L \delta}{1-ML\delta}\right) \eps_i \qquad (0< i \leq n).
\end{aligned}\right.\]
The constant $M > 0$ will be chosen (large) later on, and $\delta_0$ will also be chosen small enough so as to ensure $M L\delta \leq 1/2$ for every $\delta \leq \delta_0$. Obviously for every $i$, we have
\begin{equation}\label{upper_bound_eps_i_1}
\begin{aligned}
    \eps_{i} \leq \eps_n = \left(1 + \frac{3L \delta}{1-ML\delta}\right)^n \eps_1 &\leq \exp\left(n\log\left(1 + \frac{3L \delta}{1-ML\delta}\right)\right) \frac{2A}{L} \delta\\
    &\leq \frac{2A}L\exp\left(\frac{ \modif{6}LT}{1-M L\delta_0}\right) \delta,
\end{aligned}
\end{equation}
where we have used the concavity of the logarithm and the bound $n\delta \leq 2T$. Now we choose $M > 0$ large enough so that
\[
\frac{2A}L\exp\left({\modif{12}LT}\right) \leq M
\]
then $\delta_0 >0$ small enough so that \labelcref{angle_inequality} holds for every $0 < \delta \leq \delta_0$ and so that
\begin{equation}\label{delta0_small}
    M L\delta_0 \leq 1/2\qquad \text{and}\qquad \delta_0 M \leq \eps.
\end{equation}
With these choices of $M$ and $\delta_0$ and resuming from \labelcref{upper_bound_eps_i_1}, we have for every $i$:
\begin{equation}\label{upper_bound_eps_i_2}
    \eps_i \leq  \frac{2A}L\exp\left(\frac{\modif{6}LT}{1-M L\delta_0}\right) \delta \leq \frac{2A}L\exp\left({\modif{12}LT}\right) \delta \leq M\delta.
\end{equation}
Now from the induction relation satisfied by $\eps_i$, we have
\[\eps_{i+1} = \left(1 + \frac{3L \delta}{1-ML\delta}\right) \eps_i \geq  \left(1 + \frac{3L \delta}{1-L\eps_i}\right) \eps_i,\]
which is equivalent to the relation
\[(\eps_{i+1}-\eps_i)(1-L\eps_i) \geq 3L \delta \eps_i\]
and putting this into \labelcref{lower_bound_perturbed}, we obtain
\begin{align*}
    (\tilde x_{i+1}-\tilde x_i) \cdot F(\tilde x_i) &\geq -A\delta^2 - 2L\delta \eps_i \modif{+} 3L\delta \eps_i\\
    &\geq -A\delta^2 + L\delta\eps_1\\
    &= -A\delta^2 + 2A\delta^2\\
    &= A\delta^2 > 0.
\end{align*}
We have thus shown \labelcref{perturbed_is_increasing}. Besides, by \labelcref{upper_bound_eps_i_2} and \labelcref{delta0_small}, we have $\eps_i \leq M\delta \leq \eps$ for every $\delta\leq \delta_0$, and \labelcref{perturbed_is_close} holds true.
\end{proof}

\begin{prp}\label{ND-lipschitz}
Let $F : \R^d \to \R^d$ be a locally Lipschitz cyclically quasi-monotone vector field satisfying \labelcref{hyp_nonzero}. The family of convex sets $\C^F$ is totally ordered.
\end{prp}
\begin{proof}
Let $x \neq x'$ be two points of $\R^d$. By contradiction, we assume that they are not totally ordered, i.e. $C^F(x) \not\subseteq C^F(x')$ and $C^F(x') \not\subseteq C^F(x)$. By \Cref{convex_sets_intersect}, $\intr(C^F(x)) \cap \intr(C^F(x')) \neq \emptyset$. Let us justify that
\[\bd C^F(x) \cap \intr(C^F(x')) \neq \emptyset.\]
Take a point $y\in \intr(C^F(x)) \cap \intr(C^F(x'))$, and $z\in C^F(x') \setminus C^F(x)$. Define $y_t = (1-t)y + tz$ for $t\in [0,1]$ and $t^* = \sup\{t : y_t \in C^F(x)\}$. Since $C^F(x)$ is closed, $y_{t^*} \in \bd C^F(x)$, and since $y\in \intr(C^F(x)), z\not\in C^F(x)$ we have $t^* \in (0,1)$. Thus, $y_{t^*}$ belongs to the open segment $(y,z)$ with $y\in \intr(C^F(x')), z\in C^F(x')$, which implies that $y_{t^*} \in \intr(C^F(x'))$. We have thus found a point $x^* \coloneqq y_{t^*} \in \bd C^F(x) \cap \intr(C^F(x'))$.

The boundary $\bd C^F(x)$ is a path-connected manifold of class $\mathscr C^1$ which contains $x$ (thanks to \Cref{convex_prop_cont_vanish}) and $x^*$. We may consider a $\mathscr C^1$ curve $\gamma : [0,L] \to \bd C^F(x)$ which is parameterized by arc length such that $\gamma(0)=x$ and $\gamma(L) = x^*$. Take $\eps > 0 $ small so that $B(x^*,\eps) \subseteq C^F(x')$. We shall choose $\delta > 0$ small enough such that in particular $\delta \leq \min \{L,\eps\}/2$. We start by defining $t_0 = 0, x_0 = \gamma(t_0)= x$, then $t_1 > t_0$ and $x_1 = \gamma(t_1)$ such that $\abs{x_1-x_0}=\delta$ and continue inductively so as to define $t_0 < \ldots, t_{n-1}$ and $x_0 = \gamma(t_0),\ldots, \gamma(t_{n-1})$ up to a step $n-1$ such that $L-t_{n-1} \leq \delta$. Set finally $t_n = L$ and $x_n = \gamma(t_n)$, so that $\abs{x_n-x_{n-1}} \leq \delta$ because $\gamma$ is $1$-Lipschitz. It is not difficult to show that as $\delta \to 0$
\[\sum_{i=0}^{n-1} \abs{x_{i+1}-x_i} \sim_{\delta \to 0} L\]
but obviously $\sum_{i=0}^{n-1} \abs{x_{i+1}-x_i} \sim n\delta$ thus for $\delta$ small enough it holds $n\delta \leq 2L$. Consider $\delta$ small enough so that $\delta\leq \delta_0(\eps,\gamma([0,L]),L)$ as given in \Cref{lemma_perturbation}. Applying this lemma yields the existence of a sequence of points $\tilde x_0, \ldots, \tilde x_n$ satisfying the three items \modif{\labelcref{perturbed_initial}, \labelcref{perturbed_is_close} and \labelcref{perturbed_is_increasing}} of \Cref{lemma_perturbation}. In particular $(\tilde x_0, \ldots, \tilde x_n)$ is a strictly increasing path from $\tilde x_0 = x_0 = x$ to $\tilde x_n$. Thus, $x \prec \tilde x_n$. But $\tilde x_n \in B(x^*,\eps)\subseteq C^F(x')$ thus $\modif{\tilde x}_n \preceq_\modif{F} x'$ and $x\preceq_\modif{F} x'$: a contradiction with the fact that $C^F(x)\not\subseteq C^F(x')$.
\end{proof}

From Propositions~\ref{cond-resol} and~\ref{ND-lipschitz}, recalling that $\mathcal C^F$ being totally ordered means that $\preceq_F$ is total, we have solved \Cref{main_problem} in the case of regular and non-vanishing cyclically quasi-monotone maps, as stated in the following theorem.

\begin{thm}[Existence of potentials --- regular and non-vanishing case]\label{main_theorem}
Let $F : \R^d \to \R^d$ be a cyclically quasi-monotone map \modif{which is locally Lipschitz} and satisfies \labelcref{hyp_nonzero}. There exists a lower semi-continuous quasi-convex function $f :\R^d \to \R$ such that
\[\forall x \in \R^d, \quad F(x) \in \nop f(x).\]
\end{thm}

\section{Examples}\label{sec:examples}

In this section, we present several examples of cyclically quasi-monotone fields that we selected to highlight the sometimes \modif{odd behavior} of the convex sets (and the candidate quasi-convex functions) obtained from our construction, as well as some difficulties that arise in building a quasi-convex function generating such fields. For each example, we computed the convex sets induced by the preference relation $\preceq_F$ and we provide one or several quasi-convex functions that generate the field. \modif{It is not always true that $C^F(x) = \{ f \leq f(x)\}$ for every $x$ when $F = \nop{f}$ for some lsc quasi-convex function $f$ (see the one-dimensional case in \Cref{xmp:hedgehog}), which is why we explain the construction of $C^F(x)$.}

In most examples, we will only give a figure with the sets $C^F(x)=\{z \ : \ z \preceq_F x \}$, so let us briefly explain how to obtain them. We use essentially two ways. We first compute the set $\{z \ : \ z \prec_F x\}$ (recall that $\prec_F$ is described in \Cref{def_strict_pre-order}) by building some explicit $F$-ascending paths (this could \modif{sometimes} be very difficult but in the examples below it is not the case). Note that for $z \prec_F x$ to hold one must have $F(z)\neq \{0\}$. The passage from $\prec_F$ to $\preceq_F$ is then described in \Cref{def_pre-order}, \Cref{charac_pre-order} but we mostly used the characterization given in \Cref{pre-order_one-point_charac} \modif{or more precisely in \labelcref{one_step_description_C}.}

\begin{xmp}[Hedgehog]\label{xmp:hedgehog}
We start with an easy example where \Cref{ND-lipschitz} does not apply, but $\C^F$ is totally ordered and so \Cref{cond-resol} still \modif{gives} us a potential. On $\R^d$, $d\geq 1$, we define
\[F(x) \coloneqq \begin{dcases*}
\spn_+(x)& if $x\neq 0$,\\
\R^d& if $x=0$.
\end{dcases*}
\]
$F$ is cyclically quasi-monotone since it is the normal cone multi-map associated with the euclidean norm. We stress that $F$ does not satisfy the assumptions of \Cref{ND-lipschitz} since it is not single-valued. We could take a single-valued selection of $F$, but we cannot ask it to be both continuous and non-vanishing, since continuity would imply $F(0) = 0$ (\Cref{fig:hedgehog}).

\begin{figure}[H]
  \centering
  \includegraphics[scale=0.8]{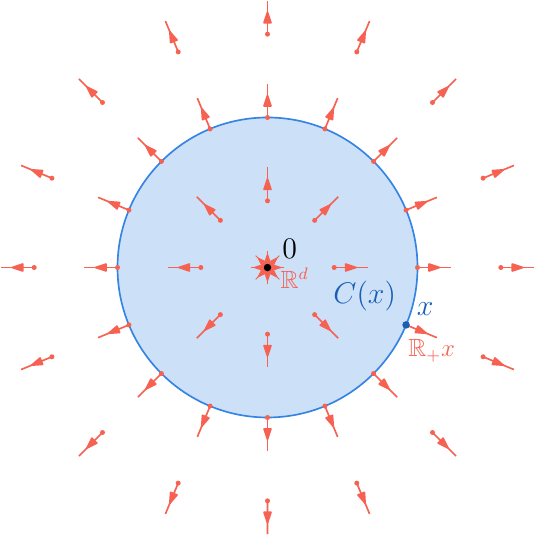}
  \caption{Hedgehog field $F$ of \Cref{xmp:hedgehog}}\label{fig:hedgehog}
\end{figure}

In dimension $d=1$, the family $\C^F$ is not totally ordered since $C^F(x) = [0,x]$ when $x\geq 0$, $C^F(x) = [x,0]$ when $x \modif{<} 0$, but \Cref{sec:1d} shows how to build one (actually several) quasi-convex functions whose normal cone multi-map contains $F$.

In dimension $d \geq 2$, we can also compute the family of convex sets $\C^F$ given by our construction: for any given $x$, identify all points $y \prec x$ (the open ball centered at $0$ and radius $\abs{x}$), then $C^F(x)$ is the intersection of all half spaces $\{z : (z-y)\cdot p \leq 0\}$ for $y\not\prec x$ and $p \in F(y)$. One may check, in this case, that $C^F(x) = \bar{B(0,\abs{x})}$ for every $x\in \R^d$, thus the family $\C^F$ is totally ordered and \Cref{cond-resol} provides the existence of many quasi-convex functions whose normal cone multi-map contains $F$.
\end{xmp}

\begin{xmp}[Plateau and closure]\label{xmp:plateau_jump}
On $\R^d$, $d\geq 1$, we set for every $x\in\R^d$:
\[f(x) \coloneqq \abs{x} \wedge ((\abs{x}-2)_+ + 1) = \begin{dcases*}
\abs{x}& if $\abs{x}\leq 1$,\\
1& if $1 < \abs{x} \leq 2$,\\
\abs{x}-1& if $\abs{x} > 2$,
\end{dcases*}\]
then
\[F \coloneqq \nop{f} \quad\text{and}\quad G \coloneqq \cl F.\]
\begin{figure}[H]
  \centering
  \includegraphics[scale=0.7]{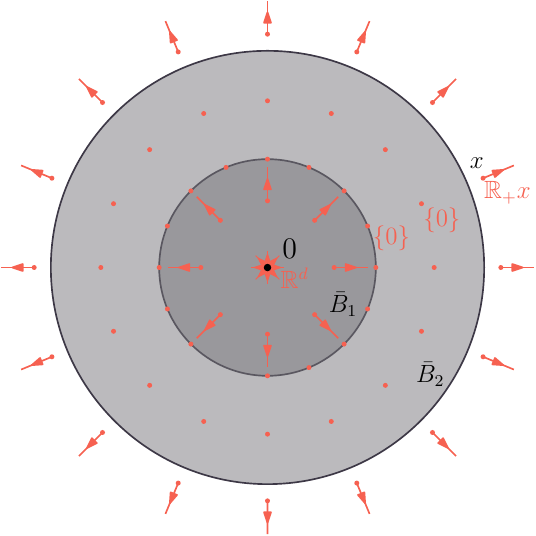}
  \caption{A field with a plateau, $F$ of \Cref{xmp:plateau_jump}}\label{fig:plateau_jump}
\end{figure}
By construction $f$ is a continuous quasi-convex function, thus $F$  and $G$ are both cyclically quasi-monotone, and $f$ is a potential for $F$ (\Cref{fig:plateau_jump}). They are, respectively, given by
\begin{align*}
    F(x) &= \begin{dcases*}
\R^d& if $x=0$,\\
\spn_+ x& if $0 <\abs{x}<1$ or $2 \leq \abs{x}$,\\
$0$& if $1 \leq \abs{x} < 2$
\end{dcases*}\\
\shortintertext{and}
G(x) &= \begin{dcases*}
\R^d& if $x=0$,\\
\spn_+ x& if $0 <\abs{x}\leq 1$ or $2 \leq \abs{x}$,\\
$0$& if $1 < \abs{x} < 2$.
\end{dcases*}
\end{align*}
The family of convex sets $\mathcal C^F$ and $\mathcal C^G$ can be explicitly computed as explained at the beginning of the section, and one may check that for every $x\in \R^d$:
\begin{equation}\label{xmp_plateau}
    C^F(x) = C_f(x) =  \begin{dcases*}
\bar B_{\abs{x}}& if $\abs{x}<1$ or $2 \leq \abs{x}$,\\
\bar B_2& if $1 \leq \abs{x} < 2$.
\end{dcases*}
\end{equation}
However, when $\abs{x} =1$, $C^G(x) = \bar B_1 \neq \bar B_2 = C^F(x)$ and $G(x) \not\subseteq \nop f(x)$. This shows that passing to the closure of a cyclically monotone operator $F$, albeit given by a continuous potential $f$, changes the family of convex sets $\modif{\mathcal C}^F$, as well as the set of admissible potentials. This happens because the original field has a plateau (an area when it is trivial), and the closure operation reduces its size.

Note that $\abs{\cdot}$ is a continuous potential for both $F$ and $G$, but if we want a potential $g$ such that $C^G(x) = C_g(x)$ and $G(x) =  \nop g(x)$ for every $x$, $g$ needs to be discontinuous at points $x \in \bd B_1$, for instance:
\[g(x) \coloneqq \begin{dcases*}
\abs{x}& if $\abs{x}\leq 1$,\\
2& if $1 < \abs{x} \leq 2$,\\
\abs{x}& if $\abs{x} > 2$.
\end{dcases*}\]
\end{xmp}

\begin{xmp}[Stadium field]\label{xmp:stadium}
In $\R^2$ let $\Sigma$ be the segment $\{0\}\times [-1,1]$ and let $d_\Sigma: \R^2\to \R^+$ be the distance from $\Sigma$ (i.e. $d_\Sigma (x)= \dist (x, \Sigma))$.
\begin{figure}[H]
  \centering
  \includegraphics[scale=0.7]{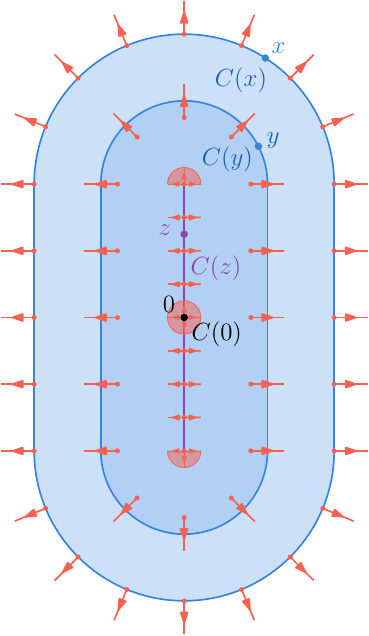}
  \caption{Stadium field of \Cref{xmp:stadium}}\label{fig:stadium}
\end{figure}

Consider the multi-valued vector field
\[F (x) \coloneqq \begin{dcases*}
\R^2& if $x=0$,\\
\R \times \{0\}& if $x$ belongs to the relative interior of $\Sigma$,\\
\R \times [0,+\infty)& if $x$ is the upper extremum of $\Sigma$,\\
\R \times (-\infty,0]& if $x$ is the lower extremum of $\Sigma$,\\
\nabla d_\Sigma (x)&  otherwise.
\end{dcases*}\]
In this case, the sets $C^F(x)$ are not essential to find a lower semicontinuous and quasi-convex potential $f$, however it is easy to see that they are like in \Cref{fig:stadium}: $C^F(0)=\{0\}$ (since $0\prec x$ for every $x\neq 0$ although it is slightly difficult to see it in the figure), $C^F(x) = \Sigma$ for every $x \in \Sigma\setminus\{0\}$ and, for $x\not \in \Sigma$, they coincide with the sub-level sets of $d_\Sigma$. One can see that computing the sets $C^F (x)$ requires, in this case, $F$-ascending paths of length at most $2$.
It is also immediate to see that $F(x)\subseteq \partial \chi_{C^F (x)}(x)$.
As admissible potential $f$,  we may consider (\modif{inspired by} \labelcref{pre_def_function_from_set_family})
\[f (x) \coloneqq \begin{dcases*}
0 & if $x=0$,\\
\modif{2}=(\dim (\Sigma)+ \modif{\len(\Sigma)/2})& if $x \in \Sigma\setminus \{0\}$,\\
2+ \LL ^2 (C^F(x))&  otherwise.
\end{dcases*}\]

\end{xmp}

\begin{xmp}[$\tfrac{3}{4}$ Hedgehog]\label{xmp:three_quarters_hedgehog}
Here, we consider a variant of \Cref{xmp:hedgehog}, namely the multi-valued vector field $F:\R^2 \multimap \R^2$ given by
\begin{figure}[H]
  \centering
  \includegraphics[scale=0.8]{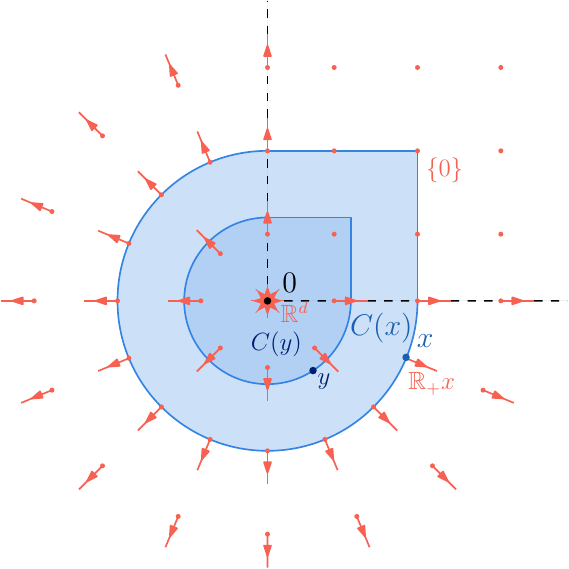}
  \caption{$\tfrac{3}{4}$ hedgehog field of \Cref{xmp:three_quarters_hedgehog}} \label{fig:three_quarters_hedgehog}
\end{figure}

\[ F(x)\coloneqq \begin{dcases*}
\spn_+(x)& if $x\neq 0$ and $\min \{ x_1 , x_2 \} \leq 0$,  \\
\R^d& if $x=0$,\\
0     & \mbox{otherwise}.
\end{dcases*}
\]
In this case, still using $F$-ascending paths of length at most 2, one can see that the sets $C^F (x)$ are the drop-shaped sets in \Cref{fig:three_quarters_hedgehog}. The family $\C^F$ is totally ordered so that \Cref{cond-resol} gives a quasi-convex and lower-semi-continuous potential $f$.
In this case, it is slightly more difficult to produce an explicit $f$ and this starts to show the necessity of an existence theorem.
\end{xmp}

\begin{xmp}[Single circle]\label{xmp:single_circle} In this example, the family $\C^F$ will not be totally ordered.
Let $d=2$ and
\[
F(x)= \begin{dcases*}
x & if $\abs{x}=1$,\\
0 & otherwise.
\end{dcases*}
\]
Note that $F(x)\in \nop f(x)$ for several quasi-convex functions $f$, for example $f(x)=|x|$, and then $F$ is cyclically quasi-monotone.
Another example of potential is
\[ f_1(x)\coloneqq \begin{dcases*}
0 & if $|x|\leq \modif{1}$,  \\
1     & \mbox{otherwise}.
\end{dcases*}
\]
Note that in the case of $f_1$ it holds $\spn_+ F(x)= \nop {f_1(x)} $ for every $x$.

We can describe the sets $C^F(x)$ explicitly. If $|x| \leq 1$ then $C^F(x)=\overline{B(0,1)}$ the closed unit ball. If $|x|>1$ then $C^F (x)$ is the sharp drop-like set in \Cref{fig:single_circle}. In this case, neither \Cref{ND-lipschitz} nor \Cref{cond-resol} apply but we explicitly showed two potentials for the vector field.
\begin{figure}[H]
  \centering
  \includegraphics[scale=0.8]{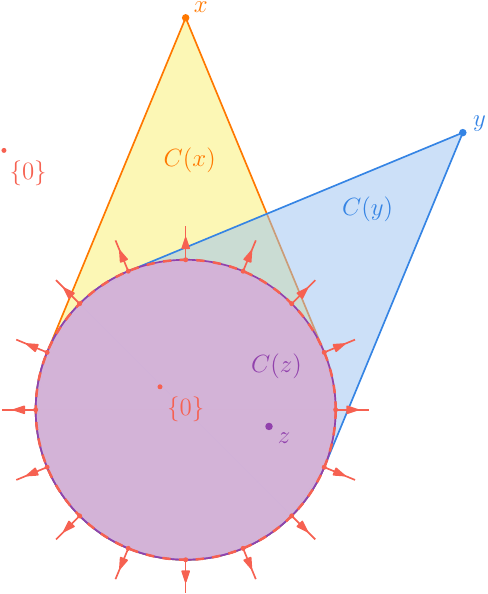}
  \caption{Single circle field of \Cref{xmp:single_circle}}\label{fig:single_circle}
\end{figure}
\end{xmp}

\begin{xmp}[$\tfrac 12$ hedgehog]\label{xmp:half_hedgehog}
On $\R^d$, $d\geq 1$, we define
\[F(x) \coloneqq \begin{dcases*}
\spn_+(x)& if $x\neq 0$ and $x_1 \leq 0$,\\
\chi_{\{x_1 \leq 0\}}\R^d& if $x=0$,\\
0 & otherwise.
\end{dcases*}
\]
\begin{figure}[H]
  \centering
  \includegraphics[scale=0.7]{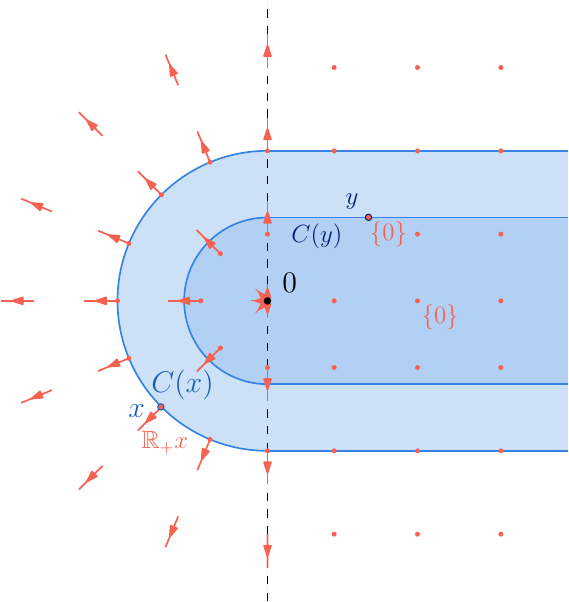}
  \caption{$\tfrac{1}{2}$ hedgehog field of \Cref{xmp:half_hedgehog}}\label{fig:half_hedgehog}
\end{figure}
In this case, the family $\C^F$ is totally ordered and is depicted in \Cref{fig:half_hedgehog}. Denoting by $\Sigma$ the half-\modif{space} $\{x \in \R^d \ : \ x_1  \geq 0\}$\modif{, an} explicit potential is given by
\[f(x)= d_\Sigma (x).\]

\end{xmp}

\begin{xmp}[$\tfrac 1 2$ norm -- $\tfrac 12$ constant]\label{xmp:half_hedgehog_half_constant}
We modify the previous example choosing
\[F(x)= (1,0,\dots,0), \ \mbox{if} \ x_1>0.\]

\begin{figure}[H]
  \centering
  \includegraphics[scale=0.8]{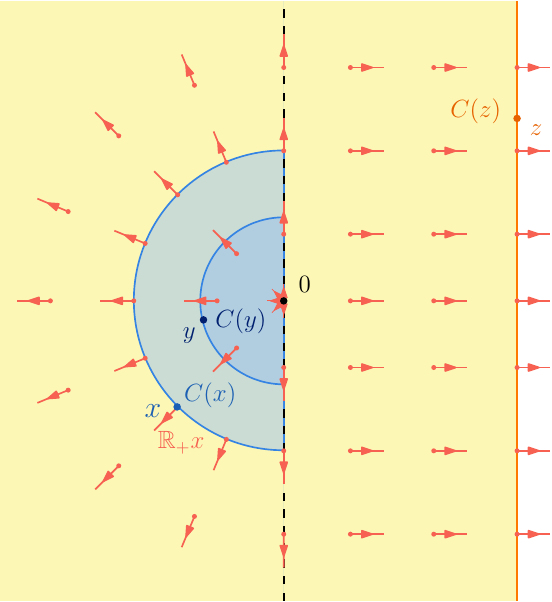}
  \caption{$\tfrac{1}{2}$ hedgehog - $\tfrac{1}{2}$ constant field of \Cref{xmp:half_hedgehog_half_constant}}\label{fig:half_edgehog_half_constant}
\end{figure}
\end{xmp}

Also in this case the family $\C^F$ is totally ordered, but more complicated. It is described as follows (see also \Cref{fig:half_edgehog_half_constant}).
If $\overline{x}_1 \leq 0$ then
\[C^F(\overline{x})= \overline{B(0, \|\overline{x}\|)} \cap \{x\ :\  x_1 \leq 0\}\]
i.e. a closed ball intersected with the \enquote{negative half-space}.
If $\overline{x}_1 >0$, then
\[C^F(\overline{x})= \{x\ :\  x_1 \leq \overline{x}_1\},\]
i.e. half-space. In this case, it is more difficult to give an explicit quasi-convex potential $f$ but we can rely on \Cref{cond-resol}. We can build an explicit potential $f$ in such a way that its sub-level sets are increasing half-disks which eventually fill the whole negative half-space, followed by a family of increasing and parallel half spaces that fill the whole $\R^2$:
\[f(x) = \begin{dcases*}
    \arctan \norm{x}& if $x_1 \leq 0$,\\
    \frac{\pi}{2}+x_1& if $x_1 > 0$.
\end{dcases*}\]

\section*{Acknowledgments}
This paper has been written during some visits of the first author at the Laboratory CEREMADE of Université Paris Dauphine-PSL and the MOKAPLAN project-team of INRIA Paris. The authors gratefully acknowledge the warm hospitality and support of these institutions.
The work of the first author is partially financed by the \textit{``Fondi di ricerca di ateneo''} of the University of Firenze and partially financed by the EU-Next Generation EU, (Missione 4, Componente 2, Investimento 1.1 \textit{Progetti di Ricerca di Rilevante Interesse Nazionale} (PRIN), CUP B53D23009310006 -- (2022J4FYNJ)). The first author is \modif{a} member of the research group GNAMPA of INdAM.
Both authors warmly thank Aris Daniilidis for several remarks, references and suggestions on an early version of the paper.
\modif{The authors would like to thank the anonymous referees for their careful reading and helpful comments, which improved the presentation of this paper.}

\sloppy
\printbibliography
\end{document}